\documentclass[11pt]{amsart}
\usepackage[utf8]{inputenc}
\usepackage{bm}
\usepackage{fontenc}
\usepackage{amsfonts}
\usepackage{amssymb}
\usepackage{amsmath}
\usepackage{amsthm}\usepackage{mathtools}
\usepackage{enumerate}
\usepackage{enumitem} 
\usepackage[pagebackref,colorlinks,linkcolor=blue,citecolor=blue,urlcolor=blue,hypertexnames=true]{hyperref}
\usepackage[pagebackref,hypertexnames=true]{hyperref}
\usepackage{enumitem}
\usepackage{mathrsfs}
\usepackage{tikz}\usetikzlibrary{calc}
\usepackage{marginnote}
\usepackage{xcolor,enumitem}
\usepackage{soul}\usepackage{tikz}
\usepackage[all,cmtip]{xy}
	
\newcommand{\N}{\mathbb{N}}
\newcommand{\C}{\mathbb{C}}

\newcommand{\eps}{\varepsilon}

\newcommand{\Manoa}{M\=anoa}

\newcommand{\Hawaii}{Hawai\kern.05em`\kern.05em\relax i}

\newcommand{\SOTh}{\mathrm{SOT}\text{-}}
\newcommand{\WOTh}{\mathrm{WOT}\text{-}}

\newcommand{\cstu}{\mathrm{C}^*_u}
\newcommand{\cstql}{\mathrm{C}^*_{\textit{ql}}}

\newtheorem*{rigprob*}{Rigidity Problem for uniform Roe Algebras}
\newtheorem*{rigprobcorona*}{Rigidity Problem for uniform Roe Coronas}

\newcommand{\cst}{\mathrm{C}^*}
\newcommand{\cstar}{$\mathrm{C}^*$}
\newcommand{\wstar}{$\mathrm{W}^*$}

\newcommand{\cZ}{\mathcal{Z}}

\newcommand{\cP}{\mathcal{P}}
\newcommand{\bbN}{\mathbb{N}}
\newcommand{\cB}{\mathcal{B}}
\newcommand{\cK}{\mathcal{K}}

\newcommand{\R}{\mathbb{R}}

\numberwithin{equation}{section}

\newtheorem{theorem}{Theorem}[section]
\newtheorem*{theorem*}{Theorem}
\newtheorem{proposition}[theorem]{Proposition}

\newtheorem*{proposition*}{Proposition}
\newtheorem{lemma}[theorem]{Lemma}
\newtheorem*{lemma*}{Lemma}
\newtheorem{corollary}[theorem]{Corollary}
\newtheorem*{corollary*}{Corollar}

\newtheorem*{fact*}{Fact}
\theoremstyle{definition}
\newtheorem{definition}[theorem]{Definition}
\newtheorem*{definition*}{Definition}
\newtheorem*{Case1}{\textbf{Case 1}}
\newtheorem*{Case2}{\textbf{Case 2}} 

\newtheorem*{acknowledgments}{Acknowledgments}
\newtheorem{claim}[theorem]{Claim}
\newtheorem*{claim*}{Claim}

\newtheorem{conjecture}[theorem]{Conjecture}
\newtheorem*{conjecture*}{Conjecture}

\theoremstyle{remark}
\newtheorem{example}[theorem]{Example}
\newtheorem*{example*}{Example}
\newtheorem{remark}[theorem]{Remark}
\newtheorem*{remark*}{Remark}

\newtheorem*{note*}{Note}
\newtheorem*{question*}{Question}

\DeclareMathOperator{\supp}{supp}

\DeclareMathOperator{\propg}{prop}

\DeclareMathOperator{\dist}{dist}

\newcommand{\cM}{\mathcal M} 
\newcommand{\cN}{\mathcal N} 
\newcommand{\bbC}{\mathbb C} 

\usepackage{enumitem}

\numberwithin{equation}{section}

\begin{document}
 
\title[On von Neumann algebras inside quasi-local algebras]{Embeddings of von Neumann algebras into uniform Roe algebras and quasi-local algebras}%

\date{\today} 

\author[Baudier]{Florent P. Baudier}
\address[F. P. Baudier]{Texas A\&M University, Department of Mathematics, College Station, TX 77843-3368, USA} 
 \email{florent@math.tamu.edu}
 \urladdr{https://www.math.tamu.edu/~florent/}

\author[Braga]{Bruno M. Braga}
\address[B. M. Braga]{PUC-Rio, Departament of Mathematics,
Gavea, Rio de Janeiro - CEP 22451-900, Brazil}
\email{demendoncabraga@gmail.com}
\urladdr{https://sites.google.com/site/demendoncabraga}
 
\author[Farah]{Ilijas Farah}

\address[I. Farah]{Department of Mathematics and Statistics\\
York University\\
4700 Keele Street\\
North York, Ontario\\ Canada, M3J 1P3\\
and 
Matemati\v cki Institut SANU\\
Kneza Mihaila 36\\
11\,000 Beograd, p.p. 367\\
Serbia}
\email{ifarah@yorku.ca}
\urladdr{https://ifarah.mathstats.yorku.ca}

\author[Vignati]{Alessandro Vignati}
\address[A. Vignati]{
Institut de Math\'ematiques de Jussieu (IMJ-PRG)\\
Universit\'e Paris Cit\'e\\
B\^atiment Sophie Germain\\
8 Place Aur\'elie Nemours \\ 75013 Paris, France}
\email{alessandro.vignati@imj-prg.fr}
\urladdr{http://www.automorph.net/avignati}

\author[Willett]{Rufus Willett}
\address[R. Willett]{University of \Hawaii~at \Manoa, 2565 McCarthy Mall, Keller 401A, Honolulu, HI 96816, USA} 
\email{rufus@math.hawaii.edu}
\urladdr{https://math.hawaii.edu/~rufus/}
 
\maketitle
 
 \begin{abstract}
We study which von Neumann algebras can be embedded into uniform Roe algebras and quasi-local algebras associated to a uniformly locally finite metric space $X$. Under weak assumptions, these \cstar-algebras contain embedded copies of $\prod_{k}\mathrm{M}_{n_k}(\C)$ for any \emph{bounded} countable (possibly finite) collection $(n_k)_k$ of natural numbers; we aim to show that they cannot contain any other von Neumann algebras. 

One of our main results shows that $L_\infty[0,1]$ does not embed into any of those algebras, even by a not-necessarily-normal $*$-homomorphism. In particular, it follows from the structure theory of von Neumann algebras that any von Neumann algebra which embeds into such algebra must be of the form $\prod_{k}\mathrm{M}_{n_k}(\C)$ for some countable (possibly finite) collection $(n_k)_k$ of natural numbers. Under additional assumptions, we also show that the sequence $(n_k)_k$ has to be bounded: in other words, the only embedded von Neumann algebras are the ``obvious'' ones.
\end{abstract} 

\tableofcontents

\section{Introduction}\label{Section.Intro}

\subsection{Uniform Roe algebras and quasi-local algebras}

Throughout this paper, $X$ is a metric space. We are interested in algebras of operators on $\ell_2(X)$, the Hilbert space of all square-summable functions from $X$ to $\C$ with its canonical Hilbert space structure and orthonormal basis $(\delta_x)_{x\in X}$. We let $\cB(\ell_2(X))$ denote the space of bounded operators on $\ell_2(X)$, and, given $A\subseteq X$, $\chi_A\in \cB(\ell_2(X))$ denotes the canonical orthogonal projection with image $\ell_2(A)$. 

In noncommutative geometry, one defines algebras of operators on $\ell_2(X)$ with the goal of encoding aspects of the geometry of $X$ in \cstar-algebraic terms. When interested in the large scale (or `coarse') geometric properties of $X$ the following two well-known \cstar-algebras are considered.

\begin{definition}\label{Defi.PropuRa}
Let $a\in\mathcal B(\ell_2(X))$. The \emph{propagation of $a$} is defined by 
\[
\mathrm{prop}(a):=\sup\{d(x,z)\mid \langle a\delta_x,\delta_z\rangle\neq 0\}.
\]
The \emph{uniform Roe algebra of~$X$}, denoted by $\cstu(X)$, is the norm closure of the $*$-algebra of operators with finite propagation.
\end{definition}

\begin{definition}\label{Defi.Quasi.Loc}
The \emph{quasi-local algebra of~$X$}, denoted by $\cstql(X)$, consists of all operators $a\in \cB(\ell_2(X))$ such that for all $\eps>0$ there is $r>0$ for which, for all $A,B\subseteq X$, $d(A,B)>r$\footnote{Throughout this paper, if $(X,d)$ is a metric space and $A,B\subseteq X$, we write $d(A,B)=\inf\{d(x,y)\mid x\in A,\ y\in B\}$; this is of course not a metric.} implies $\|\chi_Aa\chi_B\|\leq \eps$.
\end{definition}

These algebras were introduced by J. Roe to study the index theory of elliptic operators on noncompact manifolds (\cite{Roe1988,Roe1993}). Subsequently, (non-uniform) Roe algebras became important in the set up for the Baum-Connes conjecture (\cite{HigsonRoe1995,Yu:1995bv}); subsequent work (\cite{SkandalisTuYu2002,Spakula:2009tg}) made it clear that there is an equally useful version of the coarse Baum-Connes conjecture based on uniform Roe algebras. Even more recently, the quasi-local algebra has seen increased interest due to applications in index theory (\cite{Engel:2018vm,Engel:2019ux}).

Researchers in mathematical physics have also started to use uniform Roe algebras in the theory of topological materials and, in particular, topological insulators. Their importance in mathematical physics comes from the fact that, to describe a topological phase, one must choose appropriate observable algebras and symmetry types. The literature in this field has been rapidly growing and we refer the reader to \cite{Kubota2017,EwertMeyer2019,Jones2021CommMathPhys,LudewigThiang2021CommMathPhys,Bourne2022JPhys}
for more on the role of uniform Roe algebras and quasi-local algebras in mathematical physics.

Although it is elementary that $\cstu(X)$ is always a \cstar-subalgebra of $\cstql(X)$, it remains one of the biggest problems in the field to known whether these two algebras are actually the same. This entails the need for better understanding of the structure of each of these algebras. For many spaces, the situation is clear: if $X$ has Yu's property A \cite[Definition 2.1]{Yu2000}, we have that $\cstu(X)=\cstql(X)$ (\cite[Theorem 3.3]{SpakulaZhang2020JFA}, and see also \cite{SpakulaTikuisis2019}). The class of metric spaces with property A includes for instance all metric spaces with finite asymptotic dimension (\cite[Lemma 4.3]{Higson:2000dp}) such as finitely generated abelian groups and hyperbolic groups, and all amenable and all linear groups (\cite[Page 244]{GuentnerHigsonNigel2005PMIHES}).

\subsection{Goals}

For the main results of this paper, we assume that all metric spaces are \emph{uniformly locally finite} (abbreviated as \emph{u.l.f.}), that is for each $r>0$, the balls of radius $r$ have uniformly finite cardinality. This covers the most important examples such as countable discrete groups with a left-invariant proper metric, and discretizations of Riemannian manifolds with bounded sectional curvatures and injectivity radius bounded below. To avoid trivial counterexamples, we assume throughout this introduction that $X$ is infinite.

The uniform Roe algebra and quasi-local algebra of $X$ both have an unusual `hybrid' personality that sits somewhere between \cstar-algebra and von Neumann algebra theory: although they are very definitely not von Neumann algebras, they contain a copy of the von Neumann algebra $\ell_\infty(X)\subseteq \cB(\ell_2(X))$ as a \cstar-diagonal in the sense of \cite{kumjian1986c}. The presence of this `von Neumann diagonal' provides very useful tools such as strong convergence and weak compactness arguments that are not usually available to \cstar-algebraists: this has been particularly important in work on the rigidity problem for uniform Roe algebras\footnote{See \S\ref{Subsection.Methods} for more details about the rigidity problem for uniform Roe algebras.}, where the analysis of copies of the von Neumann algebra $\ell_\infty(\bbN)$ inside uniform Roe algebras and quasi-local algebras plays a pivotal role (see for example \cite{SpakulaWillett2013,BragaFarah2018Trans,WhiteWillett2017,BaudierBragaFarahKhukhroVignatiWillett2021uRaRig}). 

More generally, under very weak assumptions on $X$ (see Lemma \ref{cont lem} below) the uniform Roe algebra and quasi-local algebra contain embedded copies of the von Neumann algebra $\prod_{k}\mathrm{M}_{n_k}(\C)$ for any \emph{bounded} sequence $(n_k)_k$ of natural numbers.\footnote{Here, and throughout the paper, we use the usual terminology in operator algebras that, given a sequence of \cstar-algebras $(A_n)_n$, $\prod_n A_n$ denotes the $\ell_\infty$-sum and $\bigoplus_nA_n$ the $c_0$-sum of those algebras.} The following `folk conjecture' has thus been in the air for some time. 

\begin{conjecture}\label{big con}
The only von Neumann algebras that can embed into a uniform Roe algebra or a quasi-local algebra associated to a u.l.f.\ metric space are those of the form $\prod_k \mathrm{M}_{n_k}(\C)$, where $(n_k)_k$ is a countable (possibly finite) and bounded collection of natural numbers.\footnote{See Conjecture \ref{bigger con} below for a more precise version of this conjecture.}
\end{conjecture}

It is the purpose of this paper to study this conjecture, i.e.\ to study which von Neumann algebras can embed into uniform Roe algebras and into quasi-local algebras\footnote{It is also interesting to study which uniform Roe algebras embed into each other: this was initiated in \cite{BragaFarahVignati2019Comm}.}.

\subsection{Results}

When talking about embeddability of von Neumann algebras, the question of which topology to consider is important. Precisely, unlike the case of \cstar-algebras, where every embedding is automatically continuous in the natural (norm) topology, the category of von Neumann algebras admits singular $*$-homomorphisms that are discontinuous with respect to any of the (many!) natural von Neumann algebra topologies (see \S\ref{S.normal} for further discussion on this). It is also standard in von Neumann algebra theory to assume that all subalgebras contain the unit of the ambient algebra. Here, however, we prove the strongest possible negative results about embeddings: the embeddings are assumed to be merely $*$-algebraic (and therefore norm-continuous), and we do not assume unitality. 

We can now describe the main results of these notes. We stress that, although we state our non-embedding results in terms of the quasi-local algebra, all of them hold for the uniform Roe algebra as well, since it is included in the quasi-local algebra. 

In our investigation of which von Neumann algebras can be found inside some quasi-local algebra, the first step is to classify the \emph{abelian} von Neumann algebras with this property. Recall that an abelian von Neumann algebra $\cM$ is of the form $\mathcal{D}\oplus \ell_\infty(I)$, where $I$ is the set of all minimal projections in $\cM$, and $\mathcal{D}$ is \emph{diffuse}, i.e.\ contains no minimal projections. Note that a diffuse abelian von Neumann algebra acting on a separable Hilbert space is automatically isomorphic to $L_\infty[0,1]$ (see for example \cite[Theorem III.1.22]{Tak:TheoryI}), so the reader will lose little generality assuming that $\mathcal{D}$ is $L_\infty[0,1]$. Our first main result therefore shows that the only abelian von Neumann algebras that embed in quasi-local algebras are the obvious ones: $\ell_\infty(I)$ where $I$ is a countable (possibly finite) set.

\begin{theorem}\label{ThmLInftyQL}
Let $X$ be a u.l.f.\ metric space. There is no $*$-homomorphic embedding of a diffuse abelian von Neumann algebra into $\cstql(X)$.	
\end{theorem}
 
We point out that, although not explicitly asked in the literature, the question of whether a uniform Roe algebra could contain a subalgebra isomorphic to $L_\infty[0,1]$ was already in the air. In fact, it was even unknown up to now if a \emph{masa}, i.e.\ a maximal abelian self-adjoint subalgebra, of a uniform Roe algebra could be isomorphic to $L_\infty[0,1]$ (see e.g., \cite[\S1]{WhiteWillett2017}). Theorem \ref{ThmLInftyQL} solves this problem negatively.

Using the standard type decomposition of von Neumann agebras (see Proposition \ref{PropvNaStruc} for details), Theorem \ref{ThmLInftyQL} allows us to obtain the following corollary. 

\begin{corollary}\label{Cor.Main}
Let $X$ be a u.l.f.\ metric space and let $\cM$ be a von Neumann algebra. If there is a $*$-homomorphic embedding from $\cM$ into $ \cstql(X)$, then~$\cM$ is isomorphic to $\prod_{k}\mathrm M_{n_k}(\C)$ for some countable (possibly finite) collection $(n_k)_k$ in $\N$. 
\end{corollary}

We are then left to understand what kind of products of matrix algebras $\prod_{k}\mathrm M_{n_k}(\C)$ can be found inside a quasi-local algebra; if Conjecture \ref{big con} is true, this is possible (if and) only if $(n_k)_k$ is bounded. 

In the case that the metric space has property A, our results are already enough to solve this. Indeed, $\cstql(X)=\cstu(X)$ by \cite[Theorem 3.3]{SpakulaZhang2020JFA}, and this algebra is nuclear\footnote{Nuclearity is, by a deep theorem (see \cite[IV.3.1.5]{Black:Operator}), the correct notion of amenability in the category of \cstar-algebras. See \cite{BrownOzawa} for background on the notions of nuclearity, exactness and local reflexivity discussed here.} by \cite[Theorem 5.5.7]{BrownOzawa}. On the other hand, if $(n_k)_k$ is unbounded, $\prod_{k}\mathrm M_{n_k}(\C)$ is not exact (see for example \cite[Theorem A.1]{Ozawa:2003aa}), so cannot embed into a nuclear \cstar-algebra. Our results thus imply the following theorem.

\begin{corollary}\label{a case}
Let $X$ be a u.l.f.\ metric space with property A. Then the only von Neumann algebras that can embed in $\cstql(X)$ are the products $\prod_{k}\mathrm M_{n_k}(\C)$  for some countable (possibly finite) bounded collection $(n_k)_k$ in $\N$. \qed
\end{corollary}

%
%
%
%


 Exactness is equivalent to nuclearity for uniform Roe algebras and quasi-local algebras by  \cite[Theorem 1.1]{Sako:2020aa}, so one cannot extend Corollary~\ref{a case} in this way.

In order to move beyond the property A case, we need more delicate methods. To motivate what comes next, we recall an important definition, due originally to Yu.

\begin{definition}\label{ghost def}
An operator $a\in\cB(\ell_2(X))$ is a \emph{ghost} if for every $\eps >0$ there is a finite $F\subseteq X$ such that $\|a\delta_x\|<\epsilon$ whenever $x\notin F$. 
\end{definition}

Compact operators are natural examples of ghosts and understanding when certain ghost operators must be compact is extremely important in coarse geometry. For instance, the following are equivalent for any u.l.f.\ metric space $X$: (1) all ghosts in $\cstql(X)$ are compact, (2) all ghosts in $\cstu(X)$ are compact, (3) $\cstql(X)$ is nuclear, (4) $\cstu(X)$ is nuclear, and (5)~$X$ has Yu's property A (see  \cite[Theorem 5.5]{Li:2021aa}, which is based on \cite[Theorem 1.3]{RoeWillett2014}, \cite[Theorem~3.3]{SpakulaZhang2020JFA}, and \cite[Theorem 5.5.7]{BrownOzawa}). 

 The following is our main result regarding the embeddability of products of matrix algebras inside quasi-local algebras (and, in particular, inside uniform Roe algebras).
 
\begin{theorem}\label{ThmBlockMatrix}
Let $X$ be a u.l.f.\ metric space, and let $(n_k)_k$ be a sequence of natural numbers that tends to infinity. Then any $*$-homomorphic embedding of $\cM=\prod_k \mathrm M_{n_k}(\bbC)$ into $\cstql(X)$ which sends $\bigoplus_k \mathrm M_{n_k}(\bbC)$ to the ideal of ghost operators sends all of $\cM$ to the ideal of ghost operators. 
\end{theorem}

Theorem \ref{ThmBlockMatrix} is a corollary of a more technical result which does not require the embedding to send $\bigoplus_{k}\mathrm M_{n_k}(\C)$ to the ghost operators. Precisely, using the notation of Theorem \ref{ThmBlockMatrix}, for each $k\in \N$ and $i\in \{1,...,n_k\}$, let $e_{i,i}^k$ denote the usual diagonal matrix unit in $\mathrm M_{n_k}(\C)$. We show that if $\Phi:\prod_k \mathrm M_{n_k}(\bbC)\to \cstql(X)$ is an embedding, then the collection $\{\Phi(e_{i,i}^k)\mid k\in \N,i\in \{1,...,n_k\}\}$ is \emph{asymptotically a ghost}; this is a technical weakening of the assertion that $\Phi(1_{\cM})$ is a ghost (see Definition~\ref{Def.Asymptotically} and Remark~\ref{RemarkAsympGhosts}). 

The following immediate consequence of Theorem~\ref{ThmBlockMatrix} is worth noting. 

\begin{corollary} \label{Cor.Main.2}
Suppose that $X$ is a u.l.f.\ metric space such that $\cstql (X)$ contains no noncompact ghost projections. Assume that a von Neumann algebra $\cM$ $*$-homomorphically embeds into $\cstql(X)$ by a map sending minimal projections to compact operators. Then $\cM$ is of the form $\prod_{k}\mathrm M_{n_k}(\C)$  for some countable (possibly finite) bounded collection $(n_k)_k$ in $\N$.
\end{corollary}

The assumption that $X$ contains no noncompact ghost projections is strictly weaker than property A: see for example \cite[Theorem 5.3]{BragaChungLi2019}. Nonetheless, there are many interesting u.l.f. spaces that do not satisfy this assumption: most prominently, spaces containing expanders in a suitable sense will not satisfy it (compare \cite[pages 348-349]{HigsonLafforgueSkandalis2002GAFA}).

 Corollary \ref{Cor.Main.2} is our most complete result on Conjecture \ref{big con}. We do not know whether the geometric assumption on $X$ or the assumption on the embedding in Corollary~\ref{Cor.Main.2} can be weakened or even completely removed.

\subsection{Methods}\label{Subsection.Methods} 
Our strategy to obtain Theorems \ref{ThmLInftyQL} and \ref{ThmBlockMatrix} starts with first proving them under the stronger hypothesis that the $*$-ho\-mo\-mor\-phic embeddings are \emph{normal} (see \S\ref{S.normal} for definitions). We then get rid of this extra assumption by showing that if $L_\infty[0,1]$ $*$-homomorphically embeds into $\cB(H)$ for some separable Hilbert space $H$, then it does so by an embedding which is also normal (see Proposition \ref{P.SOT}). Our proof is completed by using a useful fact that does not seem to appear explicitly in the literature (although it is proved by assembling known results): when $\cM$ is a von Neumann algebra with no direct summands of the form $\mathrm M_n(\cN)$ for $n\geq 1$ and an infinite-dimensional abelian von Neumann algebra $\cN$, then every representation of $\cM$ on a separable Hilbert space is automatically normal (Theorem \ref{T.Normal}; the converse is also true). For expository reasons, we leave both discussions about the normality of $*$-homomorphic embeddings to the last section of this paper,~\S\ref{S.WOT}.

We now discuss the main ideas in the proofs of Theorems \ref{ThmLInftyQL} and \ref{ThmBlockMatrix} under the extra assumption of normality of the embeddings.

Our proof for Theorem \ref{ThmLInftyQL} strongly depends on working with \emph{corona algebras}. The quasi-local algebra of a u.l.f.\ metric space always contains the ideal of compact operators $\cK(\ell_2(X))$; therefore, we can look at the \emph{quasi-local corona algebra} $\cstql(X)/\mathcal K(\ell_2(X))$ (see Definition~\ref{Defi.Corona.Alg}). Using results from \cite{SpakulaZhang2020JFA} and \cite{JohPar}, we can then identify the center of $\cstql(X)/\mathcal K(\ell_2(X))$ with another important corona algebra: the Higson corona of $X$ (see Definition \ref{Defi.Higson.Corona}). Therefore, by showing that the Higson corona contains $2^{\aleph_0}$ orthogonal positive non-zero contractions (see Lemma \ref{L.HC}), we obtain that the center of $\cstql(X)/\mathcal K(\ell_2(X))$ cannot be separably represented. 

We then use Johnson--Parrott's Theorem \cite{JohPar} (see Theorem \ref{ThmJohsonParrott}) to show that if a \cstar-subalgebra $\cM\subseteq \cB(\ell_2(X))$ contained in $\cstql(X)$ is isomorphic to a diffuse abelian von Neumann algebra, then its commutant $\cM'$ must contain a copy of the nonseparably represented \cstar-algebra $\cstql(X)/\mathcal K(\ell_2(X))$. This leads to a contradiction since $\cM'$ is clearly separably represented.

To prove Theorem \ref{ThmBlockMatrix}, we must evoke the idea of subsets of $\cstql(X)$ being inside of this algebra in an ``equi-way''. We also obtain analogous equi-results for the uniform Roe algebras. To state them, we must first recall the relevant technical definitions. These were first codified in \cite[Definition 4.3]{BragaFarah2018Trans} and \cite[Definition~3.5]{BragaFarahVignati2019Comm}.

\begin{definition}\label{Defi.Equi.Sets}
Let $X$ be a metric space.
\begin{enumerate}
\item Given $\eps,r>0$, we say that $a\in \cB(\ell_2(X))$ is \emph{$\eps$-$r$-quasi-local} if for all $A,B\subseteq X$, with $d(A,B)>r$, we have that $\|\chi_Aa\chi_B\|\leq \eps$.
\item Given $\eps,r>0$, we say that $a\in \cB(\ell_2(X))$ is \emph{$\eps$-$r$-approximable} if there is $b\in \cB(\ell_2(X))$, with $\propg(b)\leq r$, such that $\|a-b\|\leq \eps$.
\item A subset $S\subseteq \cB(\ell_2(X))$ is \emph{equi-quasi-local} if for all $\eps>0$ there is $r>0$ such that every contraction $a\in S$ is $\eps$-$r$-quasi-local. $S$ is \emph{equi-approximable} if for all $\eps>0$ there is $r>0$ such that every contraction $a\in S$ is $\eps$-$r$-approximable.
\end{enumerate}
\end{definition}

The study of such ``equi-sets'' in both quasi-local and uniform Roe algebras has proven to be very useful in the study of those algebras. For instance, they have been essential in the solution of the rigidity problem (see \cite{SpakulaWillett2013,BaudierBragaFarahKhukhroVignatiWillett2021uRaRig}), as well as in the study of derivations and Hochschild homology of uniform Roe algebras (\cite{lorentz2020bounded,lorentz2021hochschild}). For the rigidity problem, it was important to show that if a weak operator topology closed subalgebra of $\cstql(X)$ (resp. $\cstu(X)$) is isomorphic to $\ell_\infty(\N)$, then its unit ball is equi-quasi-local (resp. equi-approximable); see \cite[Lemma 3.1]{SpakulaWillett2013} and \cite[Lemma 4.9]{BragaFarah2018Trans}, respectively. In this paper, we give a new version of both of these results that seems more practically useful.
 
Precisely, the next lemma is our main technical result about equi-sets in quasi-local and uniform Roe algebras. We emphasize that the metric space~$X$ in this lemma can be arbitrary (i.e., it is not required to be u.l.f.\ or even locally finite). The following is proved in \S\ref{SectionEqui}.

\begin{lemma} \label{Lemma.equi.0} 
Let $X$ be a metric space and let $\cM\subseteq \cB(\ell_2(X))$ be a $*$-subalgebra closed in the weak operator topology and containing a unit $1_{\cM}$. Suppose there is an increasing sequence $(p_n)_n$ of central projections in $ \cM$ that converges to $1_{\cM}$ in the strong operator topology and each $p_n \cM p_n$ is finite-dimensional.
\begin{enumerate}
\item\label{Lemma.equi.Item1} If $\cM\subseteq \cstql(X)$, then the unit ball of $\cM$ is equi-quasi-local.
\item\label{Lemma.equi.Item2} If $\cM\subseteq \cstu(X)$, then the unit ball of $\cM$ is equi-approximable.
\end{enumerate}
\end{lemma}

By using the standard type decomposition of von Neumann agebras in Proposition~\ref{PropvNaStruc}, Theorem \ref{ThmLInftyQL} and Lemma \ref{Lemma.equi.0} allow us to obtain the following.

\begin{theorem}\label{Thm.nVa.Equi-Approx.Unit.Ball} 
Let $X$ be a u.l.f.\ metric space. Let $\cM\subseteq \cB(\ell_2(X))$ be a weak operator topology closed \cstar-subalgebra. 
\begin{enumerate}
\item If $\cM\subseteq \cstql(X)$, then the unit ball of $\cM$ is equi-quasi-local.
\item If $\cM\subseteq \cstu(X)$, then the unit ball of $\cM$ is equi-approximable.
\end{enumerate}
\end{theorem}

Besides the equi-methods described above, our proof of Theorem \ref{ThmBlockMatrix} needs another ingredient: finite-dimensional vector measures. To put this in proper context, the rigidity problem for uniform Roe algebras asked the following: given u.l.f.\ metric spaces $X$ and $Y$ with $*$-isomorphic uniform Roe algebras, does it follow that $X$ and $Y$ are coarsely equivalent? Theorem 1.2 of \cite{BaudierBragaFarahKhukhroVignatiWillett2021uRaRig} provides a positive answer to this question; the main novelty in its proof was the study of certain finite-dimensional vector measures and the use of an atomic version of the classic Lyapunov convexity theorem. In \S\ref{SectionProdMatrix}, we further develop this method obtaining a stronger technical lemma which applies to a wider range of scenarios (see Lemma~\ref{LemmaCorollaryMeasureURA}).

\subsection*{Outline of the paper}

In \S\ref{Section.Prelim} we establish some notational preliminaries and give additional background. \S\ref{SecLInfty} shows non-existence of normal embeddings of diffuse abelian von Neumann algebras into $\cstql(X)$ using corona methods. \S\ref{SectionEqui} contains our results on equi-sets; the main techniques here are based on elementary von Neumann algebra theory and the Baire category theorem. \S\ref{SectionProdMatrix} goes into detail on asymptotic ghosts and embeddings of products of matrix algebras using vector measure techniques. Finally, \S\ref{S.WOT} gives our results on automatic normality for representations of von Neumann algebras on separable Hilbert spaces; the techniques used here are again von Neumann algebraic in character.

\section{Preliminaries} \label{Section.Prelim}
 
\subsection{Basic definitions}\label{S.uRas} 
Most of the basic definitions and terminology nee\-ded for this paper were already introduced in \S\ref{Section.Intro}. Here, we only present what is left to introduce. The \cstar-algebra $\ell_\infty(X)$ of all bounded functions from~$X$ to $\C$ is identified with the multiplication operators in $\cB(\ell_2(X))$ in the canonical way: if $a=(a_x)_{x\in X}\in \ell_\infty(X)$ and $\xi=(\xi_x)_{x\in X}\in \ell_2(X)$, then $a\xi=(a_x\xi_x)_{x\in X}$. In other words, $a\in \ell_\infty(X)$ if and only if $\propg(a)=0$. As such $\ell_\infty(X)$ identifies in a canonical way with a maximal abelian subalgebra of $\cstql(X)$. We denote by $\cK(H)$ the compact operators on a Hilbert space~$H$, and note that $\cK(\ell_2(X))$ is contained in both $\cstu(X)$ and $\cstql(X)$ as the unique minimal ideal.

We write WOT (respectively SOT, SOT$^*$) for the weak (respectively strong, strong-$^*$) operator topology on $\cB(H)$. We write ``$\SOTh\sum$'' when we want to be clear that a given sum converges in the strong operator (as opposed to norm) topology. 

We follow the standard conventions of the subject: a \emph{von Neumann algebra} will always be a concrete \cstar-algebra $\cM$ on some Hilbert space $H$ that is closed in the weak operator topology, and such that the unit of $\cM$ agrees with the unit of $\cB(H)$. A \emph{W$^*$-algebra} is an abstract \cstar-algebra that is isomorphic to some von Neumann algebra; equivalently, a W$^*$-algebra is a \cstar-algebra which is a dual Banach space. We write $\cM_1$ for the closed unit ball of a von Neumann algebra (or \cstar-algebra), and $\cZ(\cM)$ for its center. For a family $(p_s)_s$ of projections in a von Neumann algebra, $\bigvee p_s$ denotes their least upper bound, and $\bigwedge p_s$ their greatest lower bound; these always exist.

\subsection{Normal embeddings}\label{S.normal} Von Neumann algebras carry many natural topologies (see for example \cite[I.8.6]{Black:Operator}). In particular, in addition to the WOT and SOT already mentioned, there are the so-called $\sigma$-weak and $\sigma$-strong topologies (also called the ultraweak and ultrastrong topologies). These are important as they are intrinsic to the $*$-algebraic structure (see for example \cite[Corollary III.3.10]{Tak:TheoryI}), unlike the WOT and SOT, which depend on the representation; however the WOT (respectively SOT) agrees with the $\sigma$-weak (respectively $\sigma$-strong) topology on bounded sets, so often one can elide the distinction between the two. 

Fortunately, the several natural notions of continuity for $*$-homomorphisms between von Neumann algebras turn out to be the same: we record this below.

\begin{proposition}\label{n cons}
Let $\Phi:\cM\to \cN$ be a $*$-homomorphism between von Neumann algebras. Then the following are equivalent.\footnote{In each of the instances \eqref{2.normal}--\eqref{5.normal} both the domain and the codomain are considered with respect to the same topology.} 
\begin{enumerate}
\item \label{1.normal} $\Phi$ is normal: if $(a_i)_{i\in I}$ is a bounded increasing net of positive elements of $\cM$, then $\Phi(\sup_{i\in I} a_i)=\sup_{i\in I} \Phi(a_i)$. 
\item \label{6.normal} For any collection of orthogonal projections $(p_i)_i$ in $\cM$, 
\[
\textstyle\Phi(\SOTh\sum_{i\in I} p_i)=\SOTh\sum_{i\in I} \Phi(p_i).
\]
\item \label{2.normal} $\Phi$ is $\sigma$-weakly continuous.
\item \label{3.normal} $\Phi$ is $\sigma$-strongly continuous.
\item \label{4.normal} The restriction of $\Phi$ to the unit ball of $\cM$ is $\sigma$-strongly continuous. 
\item \label{5.normal} $\Phi$ is continuous with respect to the weak$^*$-topologies induced by the unique preduals (see for example \cite[III.2.4.1]{Black:Operator}) of $\cM$ and~$\cN$. 
\end{enumerate}
\end{proposition}

\begin{proof}
Parts \eqref{1.normal}, \eqref{2.normal}, \eqref{3.normal}, and \eqref{4.normal} are equivalent by \cite[III.2.2.2]{Black:Operator}. Parts \eqref{2.normal} and \eqref{5.normal} are equivalent as the weak-$*$ and $\sigma$-weak topologies are the same (see for example \cite[I.8.6.2 and III.2.4.1]{Black:Operator}). Condition \eqref{1.normal} clearly implies \eqref{6.normal}. Finally, assume that condition \eqref{6.normal} holds. Then for any normal linear functional $\psi:\cN\to \C$ and orthogonal collection of projections $(p_i)_i$, $\psi\circ \Phi(\sum p_i)=\sum \psi\circ \Phi(p_i)$. Hence $\psi\circ \Phi$ is also normal by \cite[Corollary III.3.11]{Tak:TheoryI}. As the normal linear functionals are exactly the elements of the predual of a von Neumann algebra (see \cite[Page 244]{Black:Operator}) this says that the dual map $\Phi^*:\cN^*\to \cM^*$ restricts to a map $\cN_*\to \cM_*$. Hence $\Phi$ is weak-$*$ continuous, i.e.\ satisfies \eqref{5.normal}.
\end{proof}

Note that a normal $*$-homomorphism $\Phi$ is not necessarily WOT-continuous. For example, if $H$ is an infinite-dimensional Hilbert space then the embedding of $\cB(H)$ into $\cB(H\otimes H)$ that sends $a$ to $a\otimes 1_H$ is not WOT-continuous. (This example really goes to show that WOT is not intrinsic to a given von Neumann algebra, as it depends on the ambient space.) Finally, the restriction of an isomorphism between von Neumann algebras to the unit ball of the domain is continuous in the SOT, WOT, and SOT$^*$-topologies (see \cite[\S III.2.1.14]{Black:Operator}). 

\begin{definition} 
A $*$-homomorphism from a von Neumann algebra into either $\cstu(X)$ or $\cstql(X)$ is called a \emph{normal embedding} if its kernel is trivial and it satisfies any of the conditions from Proposition \ref{n cons} above when thought of as a $*$-homomorphism from $\cM$ into $\cB(\ell_2(X))$. 
\end{definition}

\subsection{`Obvious' embeddings}

Let $X$ be a u.l.f.\ metric space. Let $(n_k)_k$ be a countable (possibly finite) collection of natural numbers, and assume that $X$ contains a sequence $(X_k)_k$ of uniformly bounded disjoint subsets such that $|X_k|=n_k$ for each $k$.\footnote{Here and throughout the paper, we write $|A|$ for the cardinality of a set $A$.} Then it is straightforward to see that $\cstu(X)$ (and therefore $\cstql(X)$) contains $\prod_k\cB(\ell_2(X_k))\cong \prod_k \mathrm{M}_{n_k}(\C)$ as a normally embedded subalgebra. For such a collection $(X_k)_k$ to exist, $(n_k)_k$ must be bounded (as $X$ is u.l.f.), but under very minor assumptions on $X$, this is the only obstruction.

\begin{lemma}\label{cont lem}
Let $X$ be an infinite u.l.f.\ metric space. Assume moreover that~$X$ either has asymptotic dimension\footnote{See for example \cite[Chapter 9]{RoeBook} for background on asymptotic dimension. We note in particular that a countable group has asymptotic dimension zero if and only if it is an increasing union of finite groups: see \cite{Smith:2006aa}.} at least one, or that $X$ is a countable group with a left-invariant bounded geometry metric. 

Then for any bounded and countable collection of natural numbers $(n_k)_k$ there exists a sequence $(X_k)_k$ of uniformly bounded disjoint subsets of $X$ such that $|X_k|=n_k$ for all $k$. In particular, $\cstu(X)$ (and therefore $\cstql(X)$) contains a normally embedded copy of $\prod_k \mathrm{M}_{n_k}(\C)$.
\end{lemma}

\begin{proof}
If $X$ has asymptotic dimension at least one, the result follows easily from \cite[Lemma 2.4]{LiWillett2017}. If $X$ is a group, fix $N\in \N$ such that $N\geq n_k$ for all $k$. Choose a set $A$ in $X$ such that $|A|=N$, and assume that the diameter of $A$ is $s$ for some $s>0$. Choose a collection $(x_k)_k$ in $X$ such that $d(x_k,x_l)>2s$ for all $k\neq l$. Then the sets $Y_k:=x_kA$ are disjoint, and all have cardinality $N$ and diameter $s$. Hence we can find subsets $X_k\subseteq Y_k$ with the desired property.
\end{proof}

These observations naturally lead to the following refinement of Conjecture \ref{big con}.

\begin{conjecture}\label{bigger con}
Let $X$ be a u.l.f.\ metric space. The only von Neumann algebras that can embed into $\cstu(X)$ and $\cstql(X)$ are those of the form $\prod_k\mathrm{M}_{n_k}(\C)$, where $(n_k)_k$ is a countable (possibly finite) bounded collection of natural numbers. Moreover, $\prod_k\mathrm{M}_{n_k}(\C)$ embeds into $\cstu(X)$ or $\cstql(X)$ if and only if $X$ contains a disjoint collection $(X_k)_k$ of uniformly bounded subsets with $|X_k|=n_k$ for all $k$. 
\end{conjecture}

Note that it is not true that $\prod_k \mathrm{M}_{n_k}(\C)$ embeds into $\cstql(X)$ for any infinite u.l.f.\ metric space $X$ and bounded sequence $(n_k)_k$: the simplest counterexample is probably $X=\{n^2\mid n\in \N\}$ with the metric inherited from~$\R$. We have $\cstql(X)=\cstu(X)=\ell_\infty(X)+\cK(\ell_2(X))$, from which it follows (for example) that $\prod_{k\in \N} \mathrm M_2(\C)$ cannot embed. Notice however that containing copies of  $\prod_{k\in \N}\mathrm M_n(\C)$, for all $n\in\N$, does not characterize  being of asymptotic zero dimension: for example, if $X=\{(i^2,j^2)\mid (i,j)\in\mathbb N^2\}$ is of asymptotic zero dimension, and $\cstu(X)$ contains copies of $\prod_{k\in \N} \mathrm{M}_n(\C)$ for all   $n\in\N$. Progress on the more refined part of the conjecture above would need a careful study of spaces of asymptotic dimension zero; we do not attempt that in this paper, but the question seems interesting.

\section{Diffuse abelian von Neumann algebras inside quasi-local algebras}\label{SecLInfty}
 
The main goal of this section is to prove Theorem \ref{ThmLInftyQL}, i.e.\ that $\cstql(X)$ cannot contain an embedded diffuse abelian von Neumann algebra, under the stronger hypothesis of the embedding being also normal. We will then show in \S\ref{SectionMakingWeakSCont} that the existence of a $*$-homomorphic embedding implies the existence of a normal one (see Proposition \ref{P.SOT}), which will then give us Theorem \ref{ThmLInftyQL} in full generality.

We now state the main result of this section. 

\begin{theorem}\label{ThmLInftyWeakStar}
If $X$ is a u.l.f.\ metric space and $\cM$ is a diffuse abelian von Neumann algebra, then there is no normal embedding of $\cM$ into $\cstql(X)$.
\end{theorem}

Recall that if $\cM$ is a von Neumann algebra (or \cstar-algebra) in $\cB(H)$, then~$\cM'$ denotes its commutant. Recall also that if $p\in \cM$ is a projection, then its \emph{central cover} (also called the \emph{central carrier} or \emph{central support}) is the projection 
\begin{equation}\label{cen cov}
z_p:=\bigwedge\{q\in \cZ(\cM) \mid p\leq q\},
\end{equation}
i.e.\ the smallest projection in the center of $\cM$ that dominates $p$: see for example \cite[III.1.1.5]{Black:Operator} or \cite[page 223]{Tak:TheoryI} for more details.

The following result seems likely to be known to experts.
 
\begin{proposition}\label{PropDifAbelianVNA}
Let $H$ be a Hilbert space and let $\cM\subseteq \cB(H)$ be a diffuse abelian von Neumann algebra. Then $\cM'\cap \cK(H)=\{0\}$. 
\end{proposition}

\begin{proof}
Suppose towards a contradiction that $\cM'\cap \cK(H)\neq \{0\}$. Hence $\cM'\cap \cK(H)$ contains a non-zero self-adjoint operator, and so it contains a non-zero finite rank projection by the spectral theorem for self-adjoint compact operators. Fix a finite rank projection $p\in\cM'$ with minimum rank, i.e., $\mathrm{rank}(p)\leq \mathrm{rank}(q)$ for all projections $q\in \cM'$. Let $z_p\in \cZ(\cM')$ be the central cover of $p$, and note that $z_p\in (\cM')'=\cM$.

Let $\sim$ and $\lesssim$ denote the Murray--von Neumann equivalence and subequivalence relations on projections in $\cM'$. Using Zorn's lemma, fix a maximal subset $Q$ of $\cM'$ consisting of mutually orthogonal projections such that $q\lesssim p$ for all $q\in Q$. Since $p$ was chosen to have minimum rank among projections in $\cM'$, it follows that $p\sim q$ for all $q\in Q$. By \cite[III.1.1.12]{Black:Operator}, we have that 
\[
z_p=\SOTh\sum_{q\in Q}q.
\]
For each pair $q,r\in Q$, as $q\sim r$ we may choose partial isometries $v_{qr}$ in~$\cM'$ such that $v_{qr}=v_{rq}^*$, $q=v_{qr}v_{rq}$, and $r=v_{rq}v_{qr}$; moreover, as $q,r\leq z_p$ we have that $v_{qr}\in z_p\cM'$ for all $q,r\in Q$. From now on, consider $z_p\cM'$ as acting on $z_pH$. Following the discussion in \cite[III.1.5.1]{Black:Operator}, we see that the von Neumann algebra $\cN$ generated by $\{v_{qr}\mid q,r\in Q\}$ is isomorphic to $\cB(\ell_2(Q))$, with commutant isomorphic to $q\cB(z_p)q$ for any $q\in Q$. Writing $N$ for the rank of each $q\in Q$, we see that $\cN'\cong \mathrm{M}_N(\C)$ for some $N\in \N$.

As $\cN\subseteq z_p \cM'\subseteq \cB(z_pH)$, the previous paragraph implies that $(z_p\cM')'$ embeds into $ \mathrm{M}_{N}(\C)$. A direct check (or compare \cite[Lemma 4.1.6]{MurphyBook}) shows then that 
\[
\mathrm{M}_N(\C)\supseteq (z_p\cM')'=z_p\cM''=z_p\cM
\]
and so $z_p\cM$ is finite dimensional. Since $\cM$ is diffuse, this gives us a contradiction.
\end{proof}

For the next result, we will need to make use of some corona algebras. We start by recalling the definition of the Higson corona of a u.l.f.\ metric space. See for example \cite[\S2.3]{RoeBook} for more background. 

\begin{definition}\label{Defi.Higson.Corona}
Let $(X,d)$ be a u.l.f.\ metric space. 

\begin{enumerate}
\item A bounded function $f\colon X\to \C$ is called a \emph{Higson function}\footnote{Higson functions are also referred to as \emph{slowly oscillating} functions.} if for all $\eps,r>0$ there is a bounded subset $Z\subseteq X$ such that $|f(x)-f(y)|\leq \eps$ for all $x,y\in X\setminus Z$ with $d(x,y)\leq r$. 
\item The $\cst$-algebra of all Higson functions is denoted by $\mathrm C_h(X)$ and the \emph{Higson corona}\footnote{In topology, ``Higson corona'' would usually refer to the maximal ideal space of this \cstar-algebra; here we refer to the \cstar-algebra itself.} of $X$ is defined as $\mathrm C_Q(X)=\mathrm C_h(X)/\mathrm C_0(X)$.
\end{enumerate}
We use the multiplication action of $\mathrm C_h(X)$ on $\ell_2(X)$ to identify $\mathrm C_h(X)$ with a \cstar-subalgebra of $\cB(\ell_2(X))$. We moreover use the identification $\mathrm C_Q(X)=\mathrm C_h(X)/(\mathrm C_h(X)\cap \cK(\ell_2(X))$ to identify $\mathrm C_Q(X)$ with a \cstar-subalgebra of the Calkin algebra $\cB(\ell_2(X))/\mathcal K(\ell_2(X))$.
\end{definition}

\begin{definition}\label{Defi.Corona.Alg}
Let $X$ be a metric space. \begin{enumerate}
\item The \emph{quasi-local corona of $X$} is $\cstql(X)=\cstql(X)/\cK(\ell_2(X))$. 
\item The \emph{uniform Roe corona of $X$} is $\mathrm{Q}^*_u(X)=\cstu(X)/\cK(\ell_2(X))$.\footnote{The uniform Roe corona of a u.l.f.\ metric space was first studied in \cite{BragaFarahVignati2018AdvMath}. We do not need it for our main results here, but include some discussion for completeness.}
\end{enumerate}
\end{definition}

The next theorem is a straightforward corollary of work of Johnson and Parrott \cite{JohPar}. It will be essential in the proofs of both Proposition \ref{PropHigCorCenterURQ} and Theorem~\ref{ThmLInftyWeakStar} below. For the statement of the next few results, let $\pi\colon\mathcal B(H)\to\mathcal B(H)/\mathcal K(H)$ denote the canonical quotient map, and recall that a \emph{masa} in a \cstar-algebra is a maximal abelian self-adjoint subalgebra. 

\begin{theorem}[Johnson and Parrott]\label{ThmJohsonParrott}
Let $H$ be a Hilbert space and $\cM\subseteq \cB(H)$ be an abelian von Neumann algebra. Then, $(\pi[\cM])'=\pi[\cM']$. In particular, if $\cM$ is a masa of $\cB(H)$, then $\pi[\cM]$ is a masa of $\cB(H)/\cK(H)$. 
\end{theorem}

\begin{proof}
We just explain how to derive the result from the results of \cite[Theorem~2.1]{JohPar}. Clearly $(\pi[\cM])'\supseteq\pi[\cM']$, so we have to show that if $b\in \cB(H)$ is such that $ab-ba\in \cK(H)$ for all $a\in \cM$, then $b\in \cM+\cK(H)$. In the terminology of \cite{JohPar}, \cite[Theorem~2.1]{JohPar} says that $\cM$ has \emph{property $P_2$} (this property is not important to us), which by \cite[Lemma~1.4]{JohPar} implies that $\cM$ has \emph{property $P_1$}. As defined on \cite[page 39]{JohPar}, property $P_1$ is exactly the desired conclusion.
\end{proof}

The following result might be of interest in its own right. Parts of it are known already (see for example \cite[Theorem 3.3]{SpakulaZhang2020JFA}), but the complete statement seems to be new.

\begin{proposition}\label{PropHigCorCenterURQ}
Let $X$ be a u.l.f.\ metric space. Then 
\[
\mathrm C_Q(X)=\cZ(\mathrm Q^*_{ql}(X))=\cZ(\mathrm Q^*_u(X)).
\]
\end{proposition}

\begin{proof} By Theorem~\ref{ThmJohsonParrott}, $\ell_\infty(X)/C_0(X)$ is a masa in $\cB(\ell_2(X))/\mathcal K(\ell_2(X))$. Since it is included in each of $\mathrm Q^*_u(X)$ and $\mathrm Q^*_{ql}(X)$, it is a masa in each of these algebras. It follows that 
$$
\cZ(\mathrm Q^*_{u}(X))=\{a\in \ell_\infty(X)/C_0(X)\mid [a,b]=0~\forall b\in \mathrm Q^*_{u}(X)\}
$$
and similarly for $\cZ(\mathrm Q^*_{ql}(X))$, whence $\cZ(\mathrm Q^*_{ql}(X))\subseteq \cZ(\mathrm Q^*_{u}(X))$.   For all $f\in \mathrm C_h(X)$ and all $a\in \cstql(X)$, \cite[Theorem 3.3]{SpakulaZhang2020JFA} implies $[f,a]\in \cK(\ell_2(X))$.  Hence $\mathrm C_Q(X)\subseteq \cZ(\mathrm Q^*_{ql}(X))\subseteq \cZ(\mathrm Q^*_{u}(X))$, and it suffices to show that $\cZ(\mathrm Q^*_{u}(X))\subseteq \mathrm C_Q(X)$.

Let $a\in \cZ(\mathrm Q^*_u(X))$. As $a\in \ell_\infty(X)/\mathrm C_0(X)=\pi[\ell_\infty(X)]$, there is $b\in \ell_\infty(X)$ such that $\pi(b)=a$.
Let us show that $b\in \mathrm C_h(X)$. We proceed by contradiction. If this is not the case, we can choose $\eps>0$, $r>0$, and sequences of distinct elements $(x_n)_n$ and $(z_n)_n$ in $X$ such that for all $n\in\N$ we have
\[
d(x_n,z_n)\leq r\quad\text{and}\quad|b(x_n)-b(z_n)|\geq \eps.
\]
Let $v\in\cB(\ell_2(X))$ be the partial isometry such that $v\delta_{x_n}=\delta_{z_n}$ for all $n\in\N$ and $v\delta_x=0$ for all $x\not\in \{x_n\mid n\in\N\}$. Since $d(x_n,z_n)\leq r$ for all $n\in\N$, $v$ has finite propagation. By definition of $v$, it also follows that
\[
|\langle(vb- bv)\delta_{x_n},\delta_{z_n}\rangle|=|b(x_n)-b(z_n)|\geq \eps
\]
for all $n\in\N$. As $(\delta_{x_n})$ converges weakly to zero, $vb-bv$ is not compact which implies that 
\[
\pi(v)a-a\pi(v)=\pi(vb-bv)\neq 0.
\]
As $v\in \cstu(X)$, this show that $a$ cannot be in the center of $\mathrm Q^*_u(X)$, which is the desired contradiction.
\end{proof}

Recall that a metric space is \emph{locally finite} if all of its bounded subsets are finite. Variants of the construction in the following lemma are well-known: compare for example the proof of \cite[Theorem 3]{Keesling:1994cr}.

\begin{lemma} \label{L.HC} 
Let $X$ be a locally finite metric space and $p\in \cstql(X)$ be an infinite rank projection. Then there is an orthogonal family of $2^{\aleph_0}$ non-zero positive contractions in $\pi(p)\mathrm C_Q(X)\pi(p)$.
\end{lemma}
 
\begin{proof}
Let $\xi_1$ be a unit vector in the image of $p$, and choose a finite subset $A_1$ of $X$ such that $\|\chi_{A_1}\xi_1-\xi_1\|\leq 2^{-1-2}$. As $\chi_{A_1}$ has finite rank and $p$ has infinite rank, there exists a unit vector $\xi_2$ that is orthogonal to $\xi_1$, that is in the image of $p$, and is such that $\|\chi_{A_1}\xi_2\|\leq 2^{-2-3}$. Hence $\|\chi_{X\setminus A_1}\xi_2-\xi_2\|\leq 2^{-2-3}$. Choose a finite subset $A_2$ of $X\setminus A_1$ such that $\|\chi_{A_2}\xi_2-\xi_2\|\leq 2^{-2-2}$. Continuing this way\footnote{This is sometimes called a \emph{sliding hump} argument: see for example \cite[page 14]{AlbiacKalton2006}.} we construct an orthonormal sequence $(\xi_n)$ in the image of $p$ and a disjoint sequence $(A_n)_n$ of finite subsets of $X$ such that
\begin{equation}
\label{EqLHQ1}
\|\xi_n-\chi_{A_n}\xi_n\|\leq 2^{-n-2}
\end{equation}
for all $n\in\N$. Passing to a subsequence (and using local finiteness of $X$), we may moreover assume that $d(A_n,A_m)>3n$ for all $n\neq m$. 

Define now $g_n:X\to [0,1]$ by 
\[
g_n(x)=\max\{0,1-d(x,A_n)/n\}.
\]
In words, each $g_n$ is identically one on $A_n$, identically zero outside the $n$-neighbourhood of $A_n$, and decreases `with slope $1/n$' in between; in particular, it is $(1/n)$-Lipschitz. As $d(A_n,A_m)>3n$ for all $n\neq m$, we therefore have that the supports of $g_n$ and $g_m$ are at least $n$ apart for all $n\neq m$. As $p$ is in $\cstql(X)$ we may thus pass to another subsequence and assume that 
\begin{equation}\label{EqLHQ2}
\|\chi_{\supp(g_n)}p\chi_{\supp(g_m)}\|\leq 2^{-n-m}
\end{equation}
for all $n\neq m$ in $\N$.

For each $M\subseteq \N$, let $g_M\colon X\to [0,1]$ be defined by 
\[
g_M(x):=\sum_{n \in M} g_n(x)
\]
(notice that given $x\in X$, this sum has at most one nonzero term, so the sum converges pointwise). In other words, looking at each $g_n$ as an element in $\ell_\infty(X)$, we have $g_M=\SOTh\sum_{n\in M}g_n$ for all $M\subseteq \N$. As each $g_n$ is $(1/n)$-Lipschitz, each $g_M$ is a Higson function. For each $M\subseteq \N$, let $h_M=\pi(g_M)$, and let $q=\pi(p)$. 

A family $S$ of infinite subsets of $\N$ is \emph{almost disjoint} if $M_1\cap M_2$ is finite for all distinct $M_1,M_2\in S$. Fix an almost disjoint family of infinite subsets of $\N$ of cardinality $2^{\aleph_0}$(see for example \cite[Proposition~9.2.2]{Fa:STCstar} for a proof that such a collection exists). We claim that the collection $(qh_Mq)_{M\in S}$ is the desired family of orthogonal non-zero positive contractions. Clearly, each $qh_Mq$ is a positive contraction. For orthogonality, since the family $S$ is almost disjoint, we only need to consider disjoint $M_1,M_2\subseteq \N$. In this case, by \eqref{EqLHQ2}, we have, for $n\in M_1$ and all $m\in M_2$, that 
\[
\|pg_npg_mp\|\leq \|g_npg_m\|\leq 2^{-n-m}.
\]
Therefore, 
\[
pg_{M_1}pg_{M_2}p=\sum_{n\in M_1}\sum_{m\in M_2}pg_npg_mp
\]
converges in norm and, as each $pg_npg_mp$ is compact, $pg_{M_1}pg_{M_2}p$ is compact. 

At last, using \eqref{EqLHQ1}, we have that 
\begin{align*}
\|pg_Mp\xi_n\|&=\|pg_M\xi_n\|\geq \|p g_n\xi_n\|-\sum_{m\neq n}\|pg_m\xi_n\|\\
&\geq \|p \xi_n\|-\|p(1-g_n) \xi_n\|-\sum_{m\neq n}\|pg_m\chi_{X\setminus A_n}\xi_n\| \geq \tfrac{1}{2}
\end{align*}
for all $M\subseteq \N$ and all $n\in\N$. Hence, $pg_Mp$ is noncompact for all infinite $M\subseteq \N$. This finishes the proof.
\end{proof}

\begin{proof}[Proof of Theorem~\ref{ThmLInftyWeakStar}] 
Fix a u.l.f.\ metric space $X$. We need to show that if $\cM$ is a diffuse abelian von Neumann algebra then there is no normal embedding of $\cM$ into $\cstql(X)$. Suppose for contradiction there is such an embedding, and abuse notation slightly by writing $\cM$ for its image. 
Let $q:=\pi(1_{\cM})$. Since $\Phi$ is normal, $\cM$ is a von Neumann subalgebra of $\mathcal B(1_\cM \ell_2(X))$ by \cite[Proposition III.3.12]{Tak:TheoryI}. Furthermore, $\cM$ is diffuse and abelian, hence the relative commutant $\cM'\cap\mathcal B(1_\cM \ell_2(X))$ contains no compact operators by Proposition~\ref{PropDifAbelianVNA}, and so $\pi$ restricts to an isomorphism between $\cM'\cap\mathcal B(1_\cM \ell_2(X))=1_{\cM}\cM'1_{\cM}$ and $q\pi(\cM')q$. Theorem \ref{ThmJohsonParrott} (applied to the von Neumann algebra $\cM$ on the Hilbert space $1_{\cM}\ell_2(X)$) implies that 
\[
q\pi[\mathcal M']q=q\pi[\cM]'q.
\]
Furthermore, since $\pi[\cM]\subseteq\mathrm Q^*_{ql}(X)$, and by Proposition~\ref{PropHigCorCenterURQ} the center of $\mathrm Q^*_{ql}(X)$ is the Higson corona $\mathrm C_Q(X)$, we have
\[
q\mathrm C_Q(X)q\subseteq q\pi[\cM]'q.
\]
As $\cM$ is diffuse, $1_{\cM}$ has infinite rank, so Lemma \ref{L.HC} applied to $p=1_\cM$ gives that $q\pi[\cM]'q$ contains an orthogonal family of $2^{\aleph_0}$ nonzero positive elements. This gives us a contradiction since $q\pi[\cM]'q$ is isomorphic to $\cM'\cap \cB(1_\cM \ell_2(X))$, a separably represented $\cst$-algebra. 
\end{proof}

\section{Equi-quasi-locality and equi-approximability}\label{SectionEqui}

We need quantitative ways of saying that certain subsets of $\cB(\ell_2(X))$ belong to $\cstql(X)$ in a uniform way, and sufficient conditions for this to hold. The present section investigates this problem: Lemma~\ref{Lemma.equi.0} from the introduction is the main technical result, which we will use to prove Theorem~\ref{Thm.nVa.Equi-Approx.Unit.Ball} at the end of this section. Although for our results regarding the embeddability of von Neumann algebras, it would be enough to obtain such results only for quasi-local algebras, we also provide versions of these results for uniform Roe algebras for the sake of completeness. 

For the definitions of $\eps$-$r$-quasi-local, $\eps$-$r$-approximable, equi-quasi-local, and equi-approximable, we refer the reader to Definition \ref{Defi.Equi.Sets}. 

The following lemma follows straightforwardly from the fact that compact subsets of metric spaces are totally bounded (cf. \cite[Lemma 4.8]{BragaFarah2018Trans}). 

\begin{lemma}\label{LemmaCompEqui}
Let $X$ be a metric space. Then any norm-compact subset of $\cstql(X)$ is equi-quasi-local, and any norm-compact subset of $\cstu(X)$ is equi-approximable. 
\end{lemma}

\begin{proof} 
The case of $\cstu(X)$ is \cite[Lemma 4.8]{BragaFarah2018Trans} and the quasi-local version follows analogously. For completeness, here is a proof: if $K\subseteq \cstql(X)$ is compact and $\eps>0$, pick a finite $A\subseteq K$ which is $\eps$-dense in $K$. As $A$ is finite, there is $r>0$ such that every element in $A$ is $\eps$-$r$-quasi-local. This implies every element in $K$ is $2\eps$-$r$-quasi-local, so we are done.
\end{proof} 

\begin{lemma} \label{L.WOT-closed} 
Let $X$ be a metric space and for each pair $\eps ,r>0$ let 
\[
\mathrm{QL}(\eps,r):=\{a\in \cB(\ell_2(X)) \mid a\text{ is }\eps\text{-}r\text{-quasi-local and }\|a\|\leq 1\}
\] 
and
\[
\mathrm{A} (\eps,r):=\{a\in \cB(\ell_2(X)) \mid a\text{ is }\eps\text{-}r\text{-approximable and }\|a\|\leq 1\}.
\] 
Both $\mathrm{QL}(\eps,r)$ and $\mathrm{A}(\varepsilon,r)$ are WOT-closed. 
\end{lemma}

\begin{proof} 

Let $(a_i)_{i\in I}$ be a net of contractions WOT-converging to some $a\in \cB(\ell_2(X))$. As WOT-closed bounded sets are norm closed, $\|a\|\leq 1$. Fix positive reals $\eps$ and $r$. 

Suppose that each $a_i$ is in $\mathrm{QL}(\eps,r)$. Fix $A,B\subseteq X$ with $d(A,B)>r$, so that $\|\chi_Aa_i\chi_B\|\leq \eps$ for all $i\in I$. As the norm is WOT-lower semicontinuous, it follows that $\|\chi_Aa\chi_B\|\leq \eps$. So, $a\in \mathrm{QL}(\eps,r)$.

Suppose now that each $a_i$ is in $\mathrm{A}(\eps,r)$. For each $i\in I$, let $b_i\in\cB(\ell_2(X))$ be such that $\propg(b_i)\leq r$ and $\|a_i-b_i\|\leq \eps$. As $(a_i)_{i\in I}$ is bounded, so is $(b_i)_{i\in I}$. As bounded subsets of $\cB(\ell_2(X))$ are WOT-precompact, by passing to a subnet we can assume that $b=\WOTh\lim_ib_i$ exists. Moreover, since each $b_i$ has propagation at most $r$, so does $b$. Finally, since $\|a_i-b_i\|\leq \eps$ for all $i\in I$, we have that $\|a-b\|\leq \eps$. This shows that $a\in \mathrm A(\eps,r)$. 
\end{proof}

We are now ready for the proof of Lemma \ref{Lemma.equi.0} from the introduction.

\begin{proof}[Proof of Lemma \ref{Lemma.equi.0}] Suppose that $X$ is a metric space and $\cM\subseteq \cB(\ell_2(X))$ is a WOT-closed $*$-subalgebra with unit $1_{\cM}$. Also suppose there is an increasing sequence $(p_n)_n$ of central projections in $ \cM$ such that each $p_n \cM p_n$ is finite-dimensional and $\SOTh\lim p_n=1_{\cM}$. Replacing $\cM$ with $\cM\oplus \C(1-1_\cM)$ and $p_n$ with $p_n\oplus (1-1_{\cM})$ if necessary, we can assume without loss of ge\-ne\-ra\-li\-ty that $1_{\cM}=1$. 

We now prove \eqref{Lemma.equi.Item1}: Suppose $\cM\subseteq \cstql(X)$. We need to prove that the unit ball $\cM_1$ of $\cM$ is equi-quasi-local. Towards a contradiction, assume the contrary. Then there is $\eps>0$ such that for all $n\in\N$ there is $a\in \cM_1$ which is not $\eps$-$n$-quasi-local. 

\begin{claim}\label{Claim.EquiTail}
For all $n,m\in\N$, there is $a\in (1-p_m)\cM$ of norm 1 which is not $\eps/2$-$n$-quasi-local. 
\end{claim}

\begin{proof}
Suppose the conclusion of the claim fails for some pair $n,m\in\N$. So, all elements in the unit ball of $(1-p_m)\cM$ are $\eps/2$-$n$-quasi-local. Since $p_m\cM p_m$ is finite dimensional, there is $r>0$ such that every element in $(p_m\cM p_m)_1$ is $\eps/2$-$r$-quasi-local (Lemma~\ref{LemmaCompEqui}); without loss of generality, assume that $r>n$. Therefore, since every $a\in \cM_1$ can be written as $a=p_m a + (1-p_m)a$, this implies that every element in $\cM_1$ is $\eps$-$r$-quasi-local. This contradicts our choice of $\eps$. 
\end{proof}

\begin{claim}\label{Claim.2.6}
For each $r\in\N$, $\mathrm{QL}(\eps/4,r)\cap \cM_1$ has empty interior with respect to the restriction of the WOT-topology to $\cM_1$. 
\end{claim}

\begin{proof}
Fix $r\in\N$ and suppose towards a contradiction that there is a WOT-open $U\subseteq \cB(\ell_2(X))$ such that $U\cap \cM_1$ is nonempty and $U\cap \cM_1\subset \mathrm{QL}(\eps/4,r)$. Fix $a\in U\cap \cM_1$. So there are $\delta>0$ and normalized $\xi_1,\ldots, \xi_k,\zeta_1,\ldots, \zeta_k\in \ell_2(X)$ for which the set
\[
B\coloneqq \bigcap_{i=1}^k\{b\in \cM_1\mid |\langle (a-b)\xi_i,\zeta_i\rangle|<\delta\} 
\]
is included in $\mathrm{QL}(\eps/4,r)$.
Since $\SOTh\lim p_n=1$, there is $m\in\N$ large enough such that 
\[
\|(1-p_m)\zeta_i\|<\frac{\delta}{2}\ \text{ and }\ \|(1-p_m)a\xi_i\|<\frac{\delta}{2}
\]
for all $i\leq k$. Let $n>r$ be such that $p_ma$ is $\eps/4$-$n$-quasi-local. By Claim \ref{Claim.EquiTail}, there is a contraction $b\in (1-p_m)\cM$ which is not $\eps/2$-$n$-quasi-local. Then, letting $c=p_ma+b$, we have that 
\[
\|c\|=\max \{\|p_ma\|,\|b\|\}\leq 1
\]
and that $a-c=(1-p_m)(a-b)$, hence 
\begin{multline*}
|\langle (a-c)\xi_i, \zeta_i\rangle|=|\langle (1-p_m)(a-b)\xi_i, \zeta_i\rangle|
\\
\leq |\langle (1-p_m)a\xi_i, \zeta_i\rangle|
+|\langle b\xi_i, (1-p_m)\zeta_i\rangle|<\delta,
\end{multline*}
 and $c\in B$. 
As $B\subseteq \mathrm{QL}(\eps/4,r)$, $c$ is $\eps/4$-$r$-quasi-local. As $p_ma$ is $\eps/4$-$n$-quasi-local, this implies that $b$ is $\eps/2$-$n$-quasi-local; contradiction.
\end{proof} 

We are now ready to complete the proof of case \eqref{Lemma.equi.Item1}. On bounded sets in~$\cM$, the WOT agrees with the ultraweak topology. Since the ultraweak topology coincides with the weak$^*$-topology associated with the predual of $\cM$, the Banach--Alaoglu theorem implies that the unit ball $\cM_1$ of $\cM$ is WOT-compact (and of course Hausdorff). By the previous claim and Lemma~\ref{L.WOT-closed}, we have that $\mathrm{QL}(\eps/4,r)\cap \cM_1$ is a closed subset with empty interior for all $r>0$. As $\cM\subseteq \cstql(X)$, we have that 
\[
\cM_1=\bigcup_{r\in\N}(\mathrm{QL}(\eps/4,r)\cap \cM_1)
\]
and that all sets on the right-hand side are closed and nowhere dense. 
This contradicts the Baire category theorem, completing the proof of \eqref{Lemma.equi.Item1}.

Proof of Lemma \ref{Lemma.equi.0}, \eqref{Lemma.equi.Item2}. We need to show that if $\cM\subseteq \cstu(X)$, then the unit ball of $\cM$ is equi-approximable. This proof follows the general strategy of the proof of \eqref{Lemma.equi.Item1}, with Claim~\ref{Claim.EquiTail} and Claim~\ref{Claim.2.6} replaced with the following two claims; the proofs are analogous to the proofs of the former claims and are left to the reader. 

\begin{claim}
For all $n,m\in\N$, there is $a\in (1_{\cM}-p_m)\cM$ of norm 1 which is not $\eps/2$-$n$-approximable. \qed
\end{claim}

\begin{claim}
For each $r\in\N$, $\mathrm A(\eps/4,r)\cap \cM_1$ has empty interior with respect to the restriction of the WOT-topology to $\cM_1$. \qed
\end{claim}
As $\cM\subseteq \cstu(X)$, we have that 
\[
\cM_1=\bigcup_{r\in\N}(\mathrm A(\eps/4,r)\cap \cM_1).
\]
As in \eqref{Lemma.equi.Item1}, this contradicts the Baire category theorem.
\end{proof}

\begin{proposition}\label{Corollary.Lemma.equi} 
Let $X$ be a countable metric space and let $\cM\subseteq\cB(\ell_2(X))$ be a WOT-closed $*$-subalgebra isomorphic to a direct product of matrix algebras.
\begin{enumerate}
\item If $\cM\subseteq \cstql(X)$, then the unit ball of $\cM$ is equi-quasi-local.
\item If $\cM\subseteq \cstu(X)$, then the unit ball of $\cM$ is equi-approximable.
\end{enumerate}
\end{proposition}

\begin{proof} 
Since $\ell_2(X)$ is separable, $\cM$ is a direct product of at most countably many matrix algebras. It therefore contains an increasing sequence $(p_n)_n$ of central projections in $ \cM$ such that each $p_n \cM p_n$ is finite-dimensional and $\SOTh\lim p_n=1_{\cM}$. The result then follows from Lemma~\ref{Lemma.equi.0}.
\end{proof} 

Our next goal is Proposition~\ref{PropvNaStruc}, which characterizes $W^*$-algebras that do not contain a diffuse abelian $W^*$-subalgebra. This seems likely to be known to experts, but we could not find a proof in the literature so include one for completeness. It is the final ingredient needed to complete the proof of our main equi-approximability result (Theorem \ref{Thm.nVa.Equi-Approx.Unit.Ball}) from the introduction.

\begin{proposition}\label{PropvNaStruc}
Assume that $\cM$ is a \wstar-algebra such that there is no normal (possibly non-unital) embedding of a diffuse abelian von Neumann algebra into $\cM$. Then $\cM$ is isomorphic to $\prod_{i\in I}\mathrm M_{n_i}(\C)$ for some collection $(n_i)_{i\in I}$ of natural numbers. 
\end{proposition}

\begin{proof} 
Using the type decomposition for von Neumann algebras (see for example \cite[III.1.4.7]{Black:Operator} or \cite[Theorem~V.1.19]{Tak:TheoryI}) and the structure theory of type I von Neumann algebras (see for example \cite[III.1.5.12 and III.1.5.13]{Black:Operator} or \cite[Theorem~V.1.27]{Tak:TheoryI}), we may write $\cM$ as a direct sum 
\[
\cM=\cM_I\oplus \cM_{II}\oplus \cM_{III}
\]
where $\cM_{I}$ is the direct product of von Neumann algebras of the form $\cB(H_\aleph)\overline{\otimes} \cN_\aleph$ with $H_\aleph$ is a Hilbert space of dimension $\aleph$ for a cardinal $\aleph$ and~$\cN_\aleph$ an abelian von Neumann algebra (possibly zero), and $\cM_{II}$ and~$\cM_{III}$ are of types II and III respectively (possibly zero). 

Now, if there is an infinite $\aleph$ such that one of the algebras $\cB(H_\aleph)\overline{\otimes} \cN_\aleph$ appearing in $\cM_I$ is non-zero then $H_\aleph$ contains an isometrically embedded copy of $L_2[0,1]$. Hence $\cM$ contains a normally embedded copy of $\cB(L_2[0,1])$, and therefore a normally embedded copy of the diffuse von Neumann algebra $L_\infty[0,1]$, which is impossible. If $\cM_{II}$ or $\cM_{III}$ are non-zero, then at least one of them contains a non-zero maximal abelian subalgebra $\cN$. As type II and type III von Neumann algebras have no minimal projections, $\cN$ is diffuse, again contradicting our assumption. Hence we may assume $\cM=\cM_I=\prod_{n\in \N} \mathrm{M}_n(\cN_n)$. 

As $\cM$ contains no abelian diffuse subalgebras, each $\cN_n$ must be of the form $\ell_\infty(I_n)$ for some set $I_n$ (possibly empty). As $\mathrm{M}_n(\ell_\infty(I_n))\cong \prod_{i\in I_n} \mathrm{M}_n(\C)$, the result follows.
\end{proof}

We are now ready for the proof of Theorem \ref{Thm.nVa.Equi-Approx.Unit.Ball} from the introduction.

\begin{proof}[Proof of Theorem \ref{Thm.nVa.Equi-Approx.Unit.Ball}] Suppose $X$ is a u.l.f.\ metric space and $\cM\subseteq \cstql(X)$ is a WOT-closed \cstar-subalgebra. Theorem \ref{ThmLInftyWeakStar} implies there is no normal embedding of a diffuse abelian von Neumann algebra into $\cM$. Proposition~\ref{PropvNaStruc} implies that $\cM$ is isomorphic to a product of matrix algebras; moreover, as $\ell_2(X)$ is separable, there can be at most countably many matrix algebras appearing in the product. Proposition~\ref{Corollary.Lemma.equi} implies that if $\cM\subseteq \cstql(X)$ then the unit ball of $\cM$ is equi-quasi-local, and that if $\cM\subseteq \cstu(X)$, then the unit ball of $\cM$ is equi-approximable.
\end{proof}

\section{Products of matrix algebras inside quasi-local algebras}\label{SectionProdMatrix}

Combining Theorem~\ref{ThmLInftyWeakStar} and Proposition~\ref{PropvNaStruc}, in order to understand which von Neumann algebras can be normally embedded inside quasi-local algebras, it suffices to focus on von Neumann algebras of the form $\prod_k\mathrm M_{n_k}(\C)$ for some countable collection $(n_k)_k$ of natural numbers. In this section, we obtain Theorem \ref{ThmBlockMatrix} from the introduction with the extra hypothesis that the embedding is also normal. We will then show in \S\ref{SectionMakingWeakSCont} that this hypothesis is satisfied automatically. 

The following is the main result of this section.

\begin{theorem}\label{ThmBlockMatrixWeakStar}
Let $X$ be a u.l.f.\ metric space, and let $(n_k)_k$ be a sequence of natural numbers that tends to infinity. Then any normal embedding of $\cM:=\prod_k \mathrm M_{n_k}(\bbC)$ into $\cstql(X)$ that sends $\bigoplus_k \mathrm M_{n_k}(\bbC)$ to the ideal of ghost operators sends all of $\cM$ to the ideal of ghost operators. 
\end{theorem}


The proof of Theorem~\ref{ThmBlockMatrixWeakStar} will proceed via a series of lemmas. The first of these is a simple observation about Hilbert spaces and can be found, for instance, in \cite[Lemma 3.1]{BaudierBragaFarahKhukhroVignatiWillett2021uRaRig}. We include its short proof here for the reader's convenience. 
 
\begin{lemma}\label{LemmaProfHalf}
Let $H$ be a Hilbert space, $p\in \cB(H)$ be a projection, and $\xi\in H$. Then $\|\xi\|= 2\|p\xi-\frac{1}{2}\xi\|$. 
\end{lemma}

\begin{proof}
Let $u=2p-1$, so $u$ is a unitary, and in particular an isometry. Hence
\[
\|\xi\|=\|u\xi\|= \|2p\xi-\xi\|=2\|p\xi-\tfrac{1}{2}\xi\|. \qedhere 
\]
\end{proof}

The next lemma we need for the proof of Theorem~\ref{ThmBlockMatrixWeakStar} is about vector measures, and is taken from \cite[Lemma~2.1]{BaudierBragaFarahKhukhroVignatiWillett2021uRaRig}\footnote{The statement of that lemma includes an extra ``$+\epsilon$'' on the right hand side of the inequality in the conclusion. However, the extra $\epsilon$ is only necessary if one wants the set~$A$ in the conclusion to be finite: see the first of the two proofs of Lemma~2.1 given in \cite{BaudierBragaFarahKhukhroVignatiWillett2021uRaRig}.}. Recall that a \emph{vector measure} is a function $\mu$ from a $\sigma$-algebra $\Sigma$ of subsets of a given set into a Banach space $E$ which is countably additive in the sense that if $(A_n)_{n}$ is a sequence of disjoint elements of $\Sigma$ then $\mu(\bigcup_nA_n)=\sum_n\mu(A_n)$, where the sum converges in norm. 
The norm on $\R^m$ in the statement of the lemma is arbitrary, and the notation $\mathrm{conv}(S)$ refers to the convex hull of a subset $S\subseteq \R^m$.

\begin{lemma}[{\cite{BaudierBragaFarahKhukhroVignatiWillett2021uRaRig}}]\label{LemmaShapleyFolkman}
Let $X$ be a set, $m\in\N$, and $\mu\colon\cP(X)\to( \R^m,\|\cdot\|)$ be a vector measure. For all $\xi\in \mathrm{conv}(\mu[\cP(X)])$ there is $A\in \cP(X)$ with 
\[
\|\xi-\mu(A)\|\leq m\sup_{x\in X}\|\mu(\{x\})\|. \eqno\qed
\]
\end{lemma} 

We will use this to establish the following result, which is closely related to \cite[Lemma 3.2]{BaudierBragaFarahKhukhroVignatiWillett2021uRaRig}.

\begin{lemma}\label{Claim4.6} 
Let $X$ be a u.l.f.\ metric space. Let $(p_s)_{s\in S}$ be an orthogonal family of projections on $\ell_2(X)$, and assume that for every $A\subseteq S$ the projection $p_A:=\sum_{s\in A}p_s$ is contained in $\cstql(X)$. 

Then for every $\gamma>0$ there is $\delta>0$ (depending on $\gamma$, the geometry of $X$, and the family $(p_s)_s$) such that if $A\subseteq S$ and $X_A:=\{x\in X\mid \|p_A \delta_x\|>\gamma\}$ then for every $x\in X_A$ there exists $s\in A$ such that $\|p_s\delta_x\|\geq \delta$. 
\end{lemma}

\begin{proof}
Fix $\eps\in(0,\gamma/8)$. Let $\cM\cong \ell_\infty(S)$ be the von Neumann algebra generated by the projections $p_s$. As any element of $\ell_\infty(S)$ can be approximated in norm by a finite linear combination of projections of the form $p_A$ for $A\subseteq S$, $\cM$ is contained in $\cstql(X)$. Let $S_1\subseteq S_2\subseteq \cdots$ be a sequence of finite subsets of $S$ with union $S$ (such sequence exists as separability of $\ell_2(X)$ implies that $S$ is countable). Applying Lemma \ref{Lemma.equi.0},\footnote{Notice that this is not the reason why we proved Lemma \ref{Lemma.equi.0}. In fact, for the current proof, \cite[Lemma 3.2]{SpakulaWillett2013} would suffice (compare also \cite[Lemma 4.9]{BragaFarah2018Trans}). The novelty in Lemma \ref{Lemma.equi.0} will be needed only for Theorem \ref{ThmBlockMatrix}.} with $p_n:=p_{S_n}$ gives $r>0$ such that $p_A$ is $\eps$-$r$-quasi-local for all $A \subseteq S$. Since $X$ is u.l.f., 
\[
m:=\sup_{x\in X}|B_{r}(x)|
\]
is finite. Fix $\delta>0$ such that $2m\delta<\gamma/8$; we claim that this $\delta$ has the required property. 

Assume otherwise for contradiction: if the conclusion of the lemma is false, we can find $A\subseteq S$ and $x\in X_A$ such that
\begin{equation*} \label{I.y(k)} 
\sup_{s\in A} \|p_s \delta_x\|<\delta.
\end{equation*}
Define a vector measure $\mu \colon \cP(A)\to \ell_2(B_r(x))$ by 
\begin{equation}\label{eq:aux-1}
\mu(B):=\chi_{B_r(x)} p_B \delta_x.
\end{equation} 
By our choice of $x$, we have that 
\[
\sup_{s\in A}\|\mu(\{s\})\|\leq\sup_{s\in A}\|p_s\delta_x\|< \delta.
\] 
Since $\dim_{\R}(\ell_2(B_r(x)))=2\dim_{\C}(\ell_2(B_r(x)))\leq 2m$, Lemma \ref{LemmaShapleyFolkman} gives $B\subseteq A$ such that 
\begin{equation}\label{Eq.ConditionOnA}
\|\mu(B)-\tfrac{1}{2}\chi_{B_r(x)}p_{A}\delta_x\|<2m\delta.
\end{equation}
By our choice of $r$, $p_B$ and $p_{A}$ are $\eps$-$r$-quasi-local. Therefore, as $p_B$ is $\varepsilon$-$r$-quasi-local and $d(x,X\setminus B_r(x))>r$, we get 
\begin{equation}\label{eq:aux-2}
\|\chi_{X\setminus B_r(x)}p_B\delta_{x}\|\leq \eps\ \text{ and }\ \|\chi_{X\setminus B_r(x)}p_{A}\delta_{x}\|\leq \eps.
\end{equation}
Let $\xi=p_{A}\delta_x$, and notice that $p_B\xi=p_B\delta_x$. Then
\begin{align*}
\|p_B\xi-\tfrac{1}{2}\xi\|=&\|p_B\delta_x-\tfrac{1}{2}p_{A}\delta_x\|\\
 \leq & \|p_B\delta_x-\chi_{B_r(x)}p_B\delta_x\|+\|\chi_{B_r(x)}p_B\delta_x-\tfrac{1}{2}\chi_{B_r(x)}p_{A}\delta_x\|\\ & + \|\tfrac{1}{2}\chi_{B_r(x)}p_{A}\delta_x-\tfrac{1}{2}p_{A}\delta_x\|
\end{align*}
The first term is, by the first part of \eqref{eq:aux-2}, not greater than $\varepsilon$. The second term is, by \eqref{Eq.ConditionOnA}, smaller than $2m\delta$, and the third term is, by the second part of \eqref{eq:aux-2}, not greater than $\varepsilon/2$. Therefore 
\[
\|p_B\xi -\tfrac 12 \xi\|\leq 2m\delta+\tfrac 32 \varepsilon<\tfrac {5\gamma}{16}.
\]
Lemma \ref{LemmaProfHalf} then implies that $\|\xi\|=\|p_{A}\delta_x\|<\gamma$. This is a contradiction since $\|p_{A}\delta_x\|>\gamma$ for all $x\in X_A$. 
\end{proof}

 The reader should compare the following definition to the usual notion of a ghost operator (Definition \ref{ghost def} from the introduction).

\begin{definition}\label{Def.Asymptotically}
Let $X$ be a u.l.f.\ metric space and let $(q_s)_{s\in S}$ be an orthogonal family of projections on $\ell_2(X)$. We say that $(q_s)_{s\in S}$ is \emph{asymptotically a ghost} (or an \emph{asymptotic ghost}) if for all $\epsilon>0$ there are finite subsets $F\subseteq X$ and $T\subseteq S$ such that 
\[
\Bigg\|\sum_{s\in S\setminus T} q_s\delta_x\Bigg\|<\epsilon
\]
for all $x\in X \setminus F$. 
\end{definition}

\begin{remark}\label{RemarkAsympGhosts}
As Definition \ref{Def.Asymptotically} is quite technical, let us make a few remarks. 
\begin{enumerate}[wide]
\item \label{ghost in} Let $q\leq r$ be projections on $\ell_2(X)$, and for $x\in X$ let $p_x$ be the projection onto the span of $\delta_x$. Then the \cstar-identity implies that
\begin{equation}\label{decr}
\|q\delta_x\|^2=\|qp_x\|^2=\|p_xqp_x\|\leq \|p_xrp_x\|=\|rp_x\|^2=\|r\delta_x\|.
\end{equation}
Hence in particular, if there is a finite subset $T\subseteq S$ (possibly just the empty set) such that $\SOTh\sum_{q\in S\setminus T}q_s$ is a ghost, then $(q_s)_{s\in S}$ is asymptotically a ghost.
\item \label{gag} If\footnote{This remark will be used in the proof of Theorem~\ref{ThmBlockMatrixWeakStar}.} $(q_s)_{s\in S}$ is an asymptotic ghost, and if each $q_s$ is itself a ghost (for example, if it has  finite rank), then $\SOTh\sum_{s\in S}q_s$ is also a ghost. Indeed, given $\epsilon>0$ let $T\subseteq S$ and $F\subseteq X$ be finite sets such that $\|\sum_{s\in S\setminus T}q_s\delta_x\|<\epsilon/2$ for all $x\not\in F$. As $\sum_{s\in T}q_s$ is a ghost, there is finite $F'\subseteq X$ such that $\|\sum_{s\in T}q_s\delta_x\|<\epsilon/2$ for all $x\not\in F'$. Hence for $x\not\in F\cup F'$, $\|\sum_{s\in S} q_s\delta_x\|<\epsilon$. 

In particular, this discussion and the point \label{ghost in} above show that asymptotic ghosts are only really interesting when the projections $q_s$ have infinite rank.
\item \label{bad ag} The converse to point \eqref{ghost in} above is false: there are asymptotic ghosts $(q_s)_{s\in S}$ such that for every finite subset $T\subseteq S$, $\SOTh\sum_{s\in S\setminus T}q_s$ is not a ghost. Indeed, let $X:=\bigsqcup_{n=1}^\infty X_n$ be the coarse space built from a sequence $(X_n)_n$ of expander graphs with associated Laplacian $\Delta_X\in \cB(\ell_2(X))$ as in the discussion on \cite[page 348]{HigsonLafforgueSkandalis2002GAFA}. Let $Y=X\times \N$ (equipped with the $\ell_1$-sum metric), and let $\Delta_Y$ be the operator that identifies with $\Delta_X$ on $\ell_2(X\times \{m\})$. Then $\Delta_Y$ is a bounded operator with propagation one, and so in $\cstu(Y)\subseteq \cstql(Y)$. Moreover, with respect to the decomposition $\ell_2(Y)=\bigoplus _{n,m\in \N}\ell_2(X_n\times \{m\})$, $\Delta_Y$ is a block diagonal operator acting on each $\ell_2(X_n\times \{m\})$ as the graph Laplacian of $X_n$. As $(X_n)_n$ is an expander,~$\Delta_Y$ has spectrum contained in $\{0\}\cup[\epsilon,\infty)$ for some $\epsilon>0$. Following the discussion on \cite[page 349]{HigsonLafforgueSkandalis2002GAFA}, the spectral projection $q$ associated to $\{0\}$ is the block operator that acts on $\ell_2(X_n\times\{m\})$ by the rank one projection with matrix
\begin{equation}\label{avg}
q_{n,m}:=\frac{1}{|X_n|}\begin{pmatrix}
1&\ldots & 1\\
\vdots& \ddots &\vdots\\
1 &\ldots & 1
\end{pmatrix}.
\end{equation}
Define $q_n:=\SOTh\sum_{m\in \N}q_{n,m}$ and set $S=\N$. The family $(q_s)_{s\in S}$ is then an asymptotic ghost, but $\SOTh\sum_{s\in T}q_s$ is not a ghost for \emph{any} nonempty subset $T$ of the index set $S$.
\item In the example from point \eqref{bad ag} above, it is also true that $\cstql(Y)$ itself contains non-trivial (i.e.\ infinite rank) ghost projections. We do not know of a u.l.f.\ space $X$ such that $\cstql(X)$ contains an asymptotic ghost, but no non-trivial ghost projections at all. Constructing such an example, or showing that none can exist, seems an interesting question.
\item On the other hand, if $(q_s)_{s\in S}$ is an asymptotic ghost such that $q_s\neq 0$ for all $s$ and $\sum_{s\in T}q_s$ is in $\cstql(X)$ for all $T\subseteq S$, then $X$ does not have property A. Thus the existence of non-trivial asymptotic ghosts in $\cstql(X)$ is an `exotic' phenomenon. We will not use this, so we just briefly sketch a proof. We can find a subset $T\subseteq S$ and a collection of disjoint finite subsets $(A_s)_{s\in T}$ of $X$ such that $a:=\SOTh\sum_{s\in T} \chi_{A_s}q_s\chi_{A_s}$ is a non-compact ghost. Corollary \ref{Corollary.Lemma.equi} implies that the family $(q_s)_{s\in T}$ is equi-quasi-local, and one can use this to show that $a$ is in $\cstql(X)$. Hence $X$ does not have property A by combining \cite[Theorem 1.3]{RoeWillett2014} and \cite[Theorem~3.3]{SpakulaZhang2020JFA}
\end{enumerate}
\end{remark}

Our next lemma 
is an analogue of \cite[Corollary 3.3]{BaudierBragaFarahKhukhroVignatiWillett2021uRaRig}, adapted to the asymptotic ghosts. 

\begin{lemma}\label{LemmaCorollaryMeasureURA}
Let $(X,d)$ be a u.l.f.\ metric space and let $(p_s)_{s\in S}$ be an orthogonal collection of projections in $\mathcal{B}(\ell_2(X))$. Consider the following three conditions on $(p_s)_{s\in S}$. 
\begin{enumerate}[label=(\roman*)]
	\item \label{Item1} The projection $\SOTh\sum_{s\in A}p_s$ is in $\cstql(X)$ for all $A\subseteq S$. 
	\item\label{Item2} The projection $\SOTh\sum_{s\in S}p_s$ is not a ghost.
	\item \label{Item3} The collection $(p_s)_{s\in S}$ is not asymptotically a ghost, 
\end{enumerate}
\begin{enumerate}
\item\label{1.LemmaCorollaryMeasureURA} Conditions \ref{Item1} and \ref{Item2} together imply that there are $\delta>0$, an infinite subset $X'\subseteq X$, and a function $f:X'\to S$ such that $\|p_{f(x)}\delta_x\|\geq \delta$ for all $x\in X'$.
\item \label{2.LemmaCorollaryMeasureURA} Conditions \ref{Item1} and \ref{Item3} together imply that there are $\delta>0$, an infinite subset $X'\subseteq X$, and a function $f:X'\to S$ such that $\|p_{f(x)}\delta_x\|\geq \delta$ for all $x\in X'$ and for every finite $T\subseteq S$ there exists finite $F\subseteq X$ such that $f(x)\in S\setminus T$ for all $x\in X'\setminus F$ (in other words, $f(x)$ tends to infinity in $S$ as $x$ tends to infinity in $X'$).
\end{enumerate}
\end{lemma}

\begin{remark}\label{no strong} This is the moment when we can see that replacing the notion of a ghost projection with that of an asymptotic ghost has its merits. 
	Let $Y$ and $(q_{s})_{s\in S}$ be the asymptotic ghost constructed in Remark \ref{RemarkAsympGhosts}, part \eqref{bad ag}. Then $\sum_{s\in S} q_s$ is not a ghost. Moreover, any function $f:X'\to S$ with the properties as in Lemma \ref{LemmaCorollaryMeasureURA} \eqref{1.LemmaCorollaryMeasureURA} necessarily takes finite image, and it therefore cannot satisfy the requirements on $f$ stated in \eqref{2.LemmaCorollaryMeasureURA}. This shows that the assumption that $(q_s)_{s\in S}$ is not an asymptotic ghost is necessary to deduce the stronger conclusion of Lemma \ref{LemmaCorollaryMeasureURA}.
	\end{remark}
 
\begin{proof}[Proof of Lemma \ref{LemmaCorollaryMeasureURA}]
In order to simplify notation, for each $A\subseteq S$ we define 
\[
p_A \coloneqq \SOTh\sum_{s\in S}p_s.
\] 

\eqref{1.LemmaCorollaryMeasureURA}
Assume that $(p_s)_s$ satisfies \ref{Item1} and \ref{Item2}. Then $p_S$ is not a ghost and therefore there must exist $\gamma>0$ such that the set
\[
X_S:=\{x\in X\mid \|p_S\delta_x\|>\gamma\} 
\]
is infinite. Lemma~\ref{Claim4.6} implies there exists $\delta>0$ such that for every $x\in X_S$ some $f(x)\in S$ satisfies $\|p_{f(x)}\delta_x\|>\delta$, so we are done with this part if we define $X':=X_S$.

\eqref{2.LemmaCorollaryMeasureURA}
Now assume \ref{Item1} and \ref{Item3} from the statement of Lemma~\ref{LemmaCorollaryMeasureURA}. Fix a nested collection $S_1\subseteq S_2\subseteq \cdots$ of finite subsets of $S$ whose union is $S$, and define $q_n:=\sum_{s\in S\setminus S_n} p_s$. As $(p_s)_s$ is not asymptotically a ghost there is $\gamma>0$ such that for every $n\in \bbN$ the set 
\[
X_n:=\{x\in X\mid \|p_n\delta_x\|>\gamma\} 
\] 
is infinite. By the monotonicity property as in line \eqref{decr} above and the fact that the sequence $(q_n)_n$ is decreasing, we see that $X_1\supseteq X_2\supseteq \cdots$. As each $X_n$ is infinite, we may choose a sequence $(x_n)_n$ of distinct elements of $X$ such that $x_n\in X_n$ for all $X$. 
Lemma~\ref{Claim4.6} gives $\delta>0$ such that for every $n$ there exists $f(x_n)\in S\setminus S_n$ which satisfies $\|p_{f(x)}\delta_x\|<\delta$. 
 Setting $X':=\{x_n\mid n\in \N\}$, we are done.
\end{proof}

Theorem \ref{ThmBlockMatrixWeakStar} will be obtained as a corollary of the following more technical result. 

\begin{theorem}\label{ThmBlockMatrixWeakStar2}
Let $X$ be a u.l.f.\ metric space, let $(n_k)_k$ be a sequence of natural numbers that converges to infinity, let $\cM:=\prod_k \mathrm M_{n_k}(\bbC)$, and let $\Phi:\cM\to \cstql(X)$ be a normal $*$-homomorphic embedding. Let $S:=\{(i,k)\in \N\times \N\mid 1\leq i\leq n_k\}$. For each $s=(i,k)\in S$, let $e_{i,i}^k$ be the corresponding diagonal matrix unit in $\mathrm M_{n_k}(\C)$, and define $q_s:=\Phi(e_{i,i}^k)$. Then $(q_s)_{s\in S}$ is asymptotically a ghost.
\end{theorem}

\begin{proof}
Assume for contradiction that $(q_s)_s$ is not asymptotically a ghost. Then, by the second part of Lemma \ref{LemmaCorollaryMeasureURA} there are $\delta>0$, an infinite subset $X'\subseteq X$, and a function $f:X'\to S$ such that 
\begin{equation}\label{del below}
\|q_{f(x)}\delta_x\|\geq \delta \quad \text{for all}\quad x\in X'.
\end{equation} 
Moreover, Lemma \ref{LemmaCorollaryMeasureURA} guarantees that $f$ can be taken so that  for every fixed value of $k$ there are only finitely many pairs $(i,k)$ in $S$. Therefore,    if we write $f(x)=(i(x),k(x))$, it follows that  
\begin{equation}\label{to infinity}
k(x)\to \infty \quad \text{as}\quad x\to\infty 
\end{equation}
(i.e. for any $K\in \N$ there exists a finite $F\subseteq X'$ such that if $x\in X'\setminus F$, then $k(x)\geq K$).
 For each $k\in \N$ and each pair $i,j\in \{1,...,n_k\}$, let $e_{i,j}^k$ denote the matrix in $\mathrm M_{n_k}(\C)$ with $1$ its $(i,j)$-entry and zero in all others, and set
 \[
 v_{i,j}^k:=\Phi(e_{i,j}^k).
 \]
 
\begin{claim}\label{cl1}
There are $\gamma,r>0$ such that for all $x\in X'$ and $j\in \{1,..., n_{k(x)}\}$ there is $z=z(x,j)\in B_r(x)$ such that $\|\chi_{\{z\}}v^{k(x)}_{j,i(x)}\delta_x\|\geq \gamma$.
\end{claim} 

\begin{proof}
By Lemma \ref{Lemma.equi.0}\footnote{This is the place where we use Lemma \ref{Lemma.equi.0}, and equi-approximability, or equi-quasi-locality results from earlier papers would not suffice.}, the family $\{v_{i,j}^k\mid k\in \N,i,j\in \{1,...,n_k\}\}$ is equi-quasi-local. So, there is $r>0$ such that each $v^k_{i,j}$ is $(\delta/2)$-$r$-quasi-local. Let $m=\sup_{x\in X} |B_r(x)|$ and set $\gamma=\delta/(2m)$; we claim this $\gamma$ has the desired property. 

Fix $x\in X'$ and $j\in \{1,...,n_{k(x)}\}$. Line \eqref{del below} gives $ \|q_{f(x)}\delta_x\|\geq \delta$; since $v^{k(x)}_{j,i(x)}$ is a partial isometry with source projection $q_{f(x)}$, we have that $ \|v^{k(x)}_{j,i(x)}\delta_x\|\geq \delta$. By our choice of $r$, we must have 
\[
\big\|\chi_{B_r(x)}v^{k(x)}_{j,i(x)}\delta_x\big\|\geq \delta/2.
\]
Therefore, by the choice of $\gamma$, there is $z\in B_r(x)$ such that 
\[
\big\|\chi_{\{z\}}v^{k(x)}_{j,i(x)}\delta_x\big\|\geq \gamma.\qedhere
\]
\end{proof}

\begin{claim}\label{cl2}
Given $\gamma,r>0$, let $\{z(x,j)\mid x\in X',j\in \{1,...,n_{k(x)}\}\}$ be as given by the previous claim. Then 
\[
\lim_{x\to \infty}|\{z(x,j)\mid j\in \{1,..., n_{k(x)}\}\}|=\infty.
\]
\end{claim}
 
\begin{proof}
 Let $N\in \N$ be arbitrary. Let $K\in \N$ be such that for all $k\geq K$, $n_k\geq N\gamma^{-2}$. Line \eqref{to infinity} gives a finite subset $F\subseteq X'$ such that $k(x)\geq K$ for all $x\in X'\setminus F$. We claim that $|\{z(x,j)\mid j\in \{1,..., n_{k(x)}\}\}|\geq N$ whenever $x\in X'\setminus F$, which will establish the claim.

Let $x\in X'\setminus F$ and $j\in \{1,...,n_{k(x)}\}$. As $\|\chi_{\{z(x,j)\}}v^{k(x)}_{j,i(x)}\delta_x\|\geq \gamma$ we have that with $s=(j,k(x))$ (using the fact that $0\leq p\leq q$ implies $\|aq\|\geq \|ap\|$ for all operators $a$)
\begin{multline}\label{qsz}
\big\|q_s\delta_{z(x,j)}\big\|
=\big\|\chi_{\{z(x,j)\}}q_s\big\|
\geq \|\chi_{\{z(x,j)\}}v_{j,i(x)}^{n(x)} (v_{j,i(x)}^{n(x)})^*\| 
\\
\geq \|\chi_{\{z(x,j)\}} v_{j,i(x)}^{n(x)} \|
\geq\big\|\chi_{\{z(x,j)\}}v^{n(x)}_{j,i(x)}\delta_x\big\|\geq \gamma.
\end{multline}
Fix $x\in X’$ and for $z\in X$ let 
\[
G=G(z):=\{i\leq n_{k(x)}\mid z(x,i)=z\}.
\]
 Using the Pythagorean theorem and line \eqref{qsz}, we have that 
\begin{align*}
	1\geq \Bigg\|\sum_{j\in G}q_{j,k(x)}\delta_z\Bigg\|^2= \sum_{i\in G}\|q_{j,k(x)}\delta_z\|^2\geq \gamma^2|G|.
\end{align*}
Hence, $|G(z)|\leq \gamma^{-2}$. 
Since $z\in X$ was arbitrary and $i\in \{1,\dots, n_{k(x)}\}$, this implies that $|\{z(x,j)\mid j\in \{1,..., n_{k(x)}\}\}|\geq n_{k(x)}\gamma^2$.
Since $n(k)\geq N \gamma^{-2}$, this set is larger than $N$ for $x\in X'\setminus F$, so we are done.
\end{proof} 

We are now ready to complete the proof of Theorem \ref{ThmBlockMatrixWeakStar2}. Indeed, Claims \ref{cl1} and \ref{cl2} combined imply that for any $N\in \N$ we can find $x\in X'$ such that $B_r(x)$ contains at least $N$ distinct points of the form $z(x,j)$. This contradicts that $X$ is u.l.f.. 
\end{proof}

Finally, we can complete this section with the proof of Theorem \ref{ThmBlockMatrixWeakStar}.

\begin{proof}[Proof of Theorem \ref{ThmBlockMatrixWeakStar}]
Suppose that $X$ is a u.l.f.\ metric space, $(n_k)_k$ is a sequence of natural numbers that tends to infinity, and that $\Phi\colon \cM=\prod_{k}\mathrm M_{n_k}(\C) \to \cstql(X)$ is a normal embedding sending $\bigoplus_{n}\mathrm M_{n}(\C)$ inside the ghost operators. We need to prove that $\Phi(1_{\cM})$ is a ghost.

With notation $q_s:=\Phi(e_{i,i}^k)$ as in the statement of Theorem \ref{ThmBlockMatrixWeakStar2} we have that the collection $(q_s)_{s\in S}$ is asymptotically a ghost. However, by assumption every $q_s$ is itself a ghost. Hence by Remark \ref{RemarkAsympGhosts}, part \eqref{gag}, $\Phi(1_{\cM})=\sum_{s\in S}q_s$ is also a ghost.
\end{proof}

\section{From embeddings to WOT-continuous embeddings}\label{SectionMakingWeakSCont}\label{S.WOT}

In \S\ref{SecLInfty} and \S\ref{SectionProdMatrix}, we proved versions of Theorems \ref{ThmLInftyQL} and \ref{ThmBlockMatrix} with the extra assumption that the embeddings being normal. In this section, we show that this extra assumption is not needed for the validity of those theorems. 

The main tool we use here are several automatic normality results for von Neumann algebras which might be of interest in their own right (at least some of them seem likely to be known to experts). For the sake of completeness, we include a characterization of when exactly a von Neumann algebra admits a non-normal representation on a separable Hilbert space, although we do not need this for our main results. 

\subsection{Embeddings of $\prod_k \mathrm{M}_{n_k}(\C)$}

It turns out that if $(n_k)_k$ tends to infinity, then any representation of $\prod_k \mathrm{M}_{n_k}(\C)$ on a separable Hilbert space is normal. This is due to Takemoto: see \cite[Theorem~1]{Takemoto}. Takemoto's result has apparently been overlooked: its proof closely resembles the proof of \cite[Theorem~V.5.1]{Tak:TheoryI}. The only difference is in the proof of Case I in \cite[Theorem~V.5.1]{Tak:TheoryI}: one has to argue that for every~$n$ the projections $p_{n,j,k}$ can be constructed for sufficiently large $j$ (in the type II case covered by \cite[Theorem~V.5.1]{Tak:TheoryI}, they exist for arbitrarily large $n$). 

As the precise statement we want is not explicit in Takemoto's paper \cite{Takemoto}, we show how to derive it.

\begin{theorem}[Takemoto]\label{takemoto the}
Let $(n_k)_k$ be a sequence of natural numbers that tends to infinity. Then every representation of $\cM:=\prod_{k\in \N} \mathrm{M}_{n_k}(\C)$ on a separable Hilbert space is normal.
\end{theorem}

\begin{proof}
Let $\pi:\cM\to \cB(H)$ be a non-normal representation on a Hilbert space $H$; we must show that $H$ is non-separable. According to Proposition \ref{n cons}, part \eqref{6.normal}, there is a collection of orthogonal projections $(p_i)_i$ in $\cM$ such that $\SOTh\sum\pi(p_i)<\pi(\SOTh\sum p_i)$. For each $k\in \N$, let $1_k\in \cM$ be the unit of the factor $\mathrm M_{n_k}(\C)$, and for each $n\in \N$, let 
\[
e_n:=\sum_{\{k\mid n_k=n\}} 1_{k},
\]
so $e_n\in \cM$ is a central projection in $\cM$ such that $e_n\cM$ is exactly the $n$-homogeneous part of $\cM$ (possibly zero). Let $p_{i,n}:=p_ie_n$, and let $q_n:=e_n-\sum_i p_{i,n}$ (only finitely many terms in the sum are non-zero, as $\{k\mid n_k=n\}$ is finite for every $n$). Then 
\begin{align}\label{first inq}
\SOTh\sum_n \pi(e_n) & =\SOTh\sum_n \Big(\pi(q_n)+\sum_i \pi(p_{i,n})\Big) \nonumber \\ & =\SOTh\sum_n \pi(q_n)+\SOTh\sum_{i,n}\pi(p_{i,n})
\end{align}
As for each $i$, $\SOTh\sum_n p_{i,n}=p_i$, we have that 
\begin{equation}\label{second inq}
\SOTh\sum_{i,n}\pi(p_{i,n})\leq \SOTh\sum_i \pi(p_i)<\pi(\SOTh\sum_i p_i),
\end{equation}
so combining lines \eqref{first inq} and \eqref{second inq} we get
\[
\SOTh\sum_n \pi(e_n)<\SOTh\sum_n \pi(q_n)+\pi\Big(\SOTh\sum_i p_i\Big).
\]
As $\SOTh\sum_i p_i=\SOTh\sum_{n,i} p_{i,n}$ and as $\SOTh\sum_n\pi(q_n)\leq \pi(\SOTh\sum_n q_n)$, this implies that 
\begin{align*}
\SOTh\sum_n \pi(e_n) & <\pi\Big(\SOTh\sum_nq_n\Big)+\pi\Big(\SOTh\sum_{n,i}p_{i,n}\Big) \\ & =\pi\Big(\SOTh\sum_{n}\Big(q_n+\sum_i p_{i,n}\Big)\Big)=\pi(1_{\cM}).
\end{align*}
Define $r:=\pi(1_{\cM})-\SOTh\sum_n\pi(e_n)$, which is a non-zero projection in $\pi(\cM)'$. Then $m\mapsto r\pi(m) r$ is a non-zero representation of $\cM$ on $rH$ that contains all the $e_n$ in its kernel. Hence \cite[Theorem 1]{Takemoto}\footnote{More precisely: Takemoto requires $e_n\neq 0$ for all $n$, and our assumption that $n_k\to\infty$ implies only that $e_n\neq 0$ for infinitely many $n$; nonetheless, the same proof as of \cite[Theorem 1]{Takemoto} gives the result (compare also \cite[Remark on page 575]{Takemoto}, which makes a related point).} implies that $rH$ is non-separable.
\end{proof}

\begin{proof}[Proof of Theorem \ref{ThmBlockMatrix}]
	Suppose $X$ is a u.l.f.\ metric space, $(n_k)_k$ is a sequence of natural numbers that tends to infinity, and $\Phi\colon\prod_k \mathrm M_{n_k}(\bbC) \to \cstql(X)$ is a $*$-homomorphic embedding which sends $\bigoplus_k \mathrm M_{n_k}(\bbC)$ to the ideal of ghost operators. By Theorem~\ref{T.Normal}, $\Phi$ is normal, and Theorem~\ref{ThmBlockMatrixWeakStar2} implies that $\Phi$ sends all of $\cM$ to the ideal of ghost operators. 
\end{proof}

\begin{remark}
It is also possible to adapt the proof of {\cite[Theorem 4.3]{BragaFarahVignati2019Comm}} to show that if there is a norm-continuous $*$-homomorphic embedding of $\prod_k \mathrm M_{n_k}(\bbC)$ into $\cstql(X)$, then there is a similar embedding that is also normal; this would be good enough for our results. We chose here to go through Takemoto's theorem instead as it seemed more conceptual to rely on a very general von Neumann algebra result than on something that seems special to uniform Roe algebras.
\end{remark}

\subsection{Embeddings of $L_\infty(Z,\mu)$}\label{S.Linfty}

In this subsection we show that if there exists a (possibly non-normal) embedding of $L_\infty(Z,\mu)$ into $\cB(H)$ for some separable $H$, there is a non-trivial corner of $L_\infty(Z,\mu)$ on which the embedding is normal. 

We will actually prove this in much more generality. The following is the main result of this subsection.

\begin{proposition} \label{P.SOT} 
Let $\cM$ be a von Neumann algebra, and let $\pi:\cM\to \cB(H)$ be a (not necessarily normal, not necessarily faithful, and not necessarily unital) representation of $\cM$ on a separable Hilbert space. Then there exists a non-zero projection $r\in \cM$ such that the restriction of $\pi$ to the corner $r\cM r$ is normal.
\end{proposition}

\begin{example}
Separability of $H$ is necessary for Proposition \ref{P.SOT} to hold. Indeed, let $\cM=L_\infty[0,1]$, and let $\pi:\cM\to \cB(H)$ be the direct sum of all one-dimensional representations. No one-dimensional representation of a diffuse abelian von Neumann algebra is normal (we will show this in the proof of Theorem \ref{T.Normal} below). As any corner of $\cM$ is a diffuse abelian von Neumann algebra, the conclusion of Proposition \ref{P.SOT} fails. Notice that, if one considers a single one-dimensional representation $\rho$ of $\mathcal M$, then the only possible $r$ making Proposition~\ref{P.SOT} true has to belong to the kernel of $\rho$.
\end{example}

For the proof of Proposition~\ref{P.SOT}, recall that a projection $p$ in a von Neumann algebra is called \emph{countably decomposable}, or \emph{$\sigma$-finite}, if any family of orthogonal subprojections of $p$ is countable.

\begin{proof}[Proof of Proposition~\ref{P.SOT}] 
Using \cite[III.1.2.6]{Black:Operator} there exists a family of $(c_i)_i$ of mutually orthogonal countably decomposable projections in $\cM$ such that $\sum_i c_i=1_{\cM}$ (we really only need that some $c_i$ is non-zero). Hence replacing $\cM$ with $c_i\cM c_i$ for any $i$ such that $c_i\neq 0$, we may assume that $\cM$ is countably decomposable.

As in the proof of Proposition \ref{PropvNaStruc} we may write $\cM$ as a direct sum 
\[
\cM=\cM_I\oplus \cM_{II}\oplus \cM_{III}
\]
where $\cM_{I}$ is the direct product of von Neumann algebras of the form $\cB(H_\aleph)\overline{\otimes} \cN_\aleph$ with $H_\aleph$ is a Hilbert space of dimension $\aleph$ for a cardinal $\aleph$ and $\cN_\aleph$ an abelian von Neumann algebra (possibly zero), and $\cM_{II}$ and $\cM_{III}$ are of types II and III respectively (possibly zero). If any of the summands $\cM_{III}$, $\cM_{II}$, or $\cB(H_\aleph)\overline{\otimes} \cN_\aleph$ with $\aleph$ infinite are non-zero, then \cite[Theorem~V.5.1]{Tak:TheoryI} implies that the restriction of $\pi$ to that summand is normal, and we are done.

Therefore $\cM$ has a corner which is $n$-homogeneous for some $n$, and by replacing $\cM$ with this corner we may assume that $\cM$ is of the form $\mathrm{M}_n(\C)\overline{\otimes}\cN$ for an abelian von Neumann algebra $\cN$. In particular, we may assume that~$\cM$ has a normal tracial state, say $\tau$.

Let $(p_j)_{j\in J}$ be a family of mutually orthogonal non-zero projections in $\cB(H)$ that is maximal with respect to the following condition:
\begin{enumerate}
\item[($*$)] \label{star}There exists a (countable, as $\cM$ is countably decomposable) family of mutually orthogonal non-zero projections $(q_{n,j})_{n\in \N}$ in $\cM$ such that $p_j=\pi(\SOTh\sum_{n} q_{n,j})-\SOTh\sum_n \pi(q_{n,j})$ (and this difference is non-zero).
\end{enumerate}
Of course, the family $(p_j)$ might be empty. As $H$ is separable, we may assume that $J$ is a subset of $\N$. Now, for each $j\in J$, choose $n(j)\in \N$ such that 
\[
\tau\Bigg(\SOTh\sum_{n\geq n(j)} q_{n,j}\Bigg)<2^{-j-2}
\]
($n(j)$ exists by normality of $\tau$). For each $j\in J$, define $q_j:=\SOTh\sum_{n\geq n(j)}q_{n,j}$, and define $q:=\bigvee_{j\in J} q_j$\footnote{The different $q_j$ need not be orthogonal, so this is an honest supremum, not a sum.}; we claim $r:=1_{\cM}-q$ has the desired property that the restriction of $\Phi$ to $r\cM r$ is normal. 

We first show that $r$ is non-zero. Note that for any finite $F\subseteq J$, $\tau(\bigvee_{j\in F}q_j)\leq \tau( \sum_{j\in F}q_j)$ by repeated applications of \cite[III.1.1.3]{Black:Operator}. Applying the definition of normality to the increasing net $(\bigvee_{j\in F} q_j)_{F\subseteq J\text{ finite}}$ gives 
\begin{align*}
\tau(q) & =\tau\Bigg(\lim_F \bigvee_{j\in F}q_j\Bigg)=\lim_F \tau\Big(\bigvee_{j\in F}q_j\Big)\leq \lim_F \sum_{j\in F}\tau(q_j)= \sum_{j\in J}\tau(q_j) \\ & \leq \sum_{j\in \N} 2^{-j-2} =1/2.
\end{align*}
Hence in particular, $\tau(r)\geq 1/2$, so $r$ is non-zero. 

We finally claim that $\pi$ restricted to $r\cM r$ is normal. The definition of $p_j$ (see ($*$) above) implies that $\pi(q_j)\geq p_j$. Hence for any $j$, $\pi(q)\geq \pi(q_j)\geq p_j$. Hence $1_H-\pi(q)$ is orthogonal to all the $p_j$. As $\pi(r)=\pi(1_{\cM}-q)\leq 1_H-\pi(q)$, this implies that $\pi(r)$ is orthogonal to all the $p_j$ too. Assume for contradiction that $\pi$ is not normal when restricted to $r \cM r$. Then by condition \eqref{6.normal} in Proposition \ref{n cons}, there exists a family of non-zero mutually orthogonal projections $(q_{p,n})_{n\in \N}$ in $r\cM r$ such that 
\[
p:=\pi\Bigg(\SOTh\sum_{n} q_{p,n}\Bigg)-\SOTh\sum_{n}\pi(q_{p,n})\neq 0.
\]
However, $p\leq \pi(r)$, so $p$ is orthogonal to all the $p_j$, contradicting maximality of the family $(p_j)$. 
\end{proof}

We are now ready to complete the remaining proofs of the theorems from the introduction.

\begin{proof}[Proof of Theorem \ref{ThmLInftyQL}] 
Suppose that $X$ is a u.l.f.\ metric space. Let $\cN$ be a diffuse abelian von Neumann algebra that embeds into $\cstql(X)$. Proposition~\ref{P.SOT} then implies that there is a normal embedding of $pL_\infty(Z,\mu)p$ into $\cstql(X)$, where $p$ is a non-zero projection in $\cN$. Since each non-zero corner of $\cN$ is itself a diffuse abelian von Neumann algebra, we have a normal embedding of a diffuse abelian von Neumann algebra into $\cstql(X)$. This contradicts Theorem~\ref{ThmLInftyWeakStar}, which asserts that there are no such normal embeddings. 
\end{proof} 

\begin{proof}[Proof of Corollary \ref{Cor.Main}] 
Suppose $X$ is a u.l.f.\ metric space and $\cM$ is a von Neumann algebra that embeds into $\cstql(X)$. By Theorem \ref{ThmLInftyQL} $\cM$ has no embedded diffuse abelian $\WOTh$closed subalgebras, hence Proposition \ref{PropvNaStruc} and separability of $\ell_2(X)$ imply that $\cM$ is isomorphic to $\prod_k\mathrm M_{n_k}(\C)$ for some sequence $(n_k)_{k\in\N}\subseteq \N$. 
\end{proof}

\begin{proof}[Proof of Corollary \ref{Cor.Main.2}] 
	Assume that $X$ is a u.l.f.\ metric space such that $\cstql (X)$ contains no noncompact ghost projections and a von Neumann algebra $\cM$ $*$-homomorphically embeds into $\cstql(X)$ by a map sending minimal projections to compact operators. Corollary~\ref{Cor.Main} implies that $\cM$ is isomorphic to $\prod_k\mathrm M_{n_k}(\C)$ for some sequence $(n_k)_{k\in\N}\subseteq \N$. It follows that the embeddings sends $\bigoplus_k\mathrm M_{n_k}(\C)$ to compact operators. 
	
	Assume towards a contradiction that the sequence $(n_k)_k$ is unbounded. Applying Theorem \ref{ThmBlockMatrix} to the composition of the embeddings, we conclude that all of $\prod_k \mathrm M_{n_k}(\C)$ is sent to ghost operators. This contradicts the assumption that all ghosts in $\cstql(X)$ are compact. 
\end{proof}

\subsection{A general characterization}

Finally in this section, we include a characterization of those von Neumann algebras that admit a non-normal representation on a separable Hilbert space. We also show that any von Neumann algebra that can be represented on a separable Hilbert space has separable predual. These results are included for the sake of completeness: they are not used for any of our results on embeddings into quasi-local algebras.

\begin{theorem} \label{T.Normal} 
The following are equivalent for every von Neumann algebra~$\cM$. 
\begin{enumerate}
\item \label{tn1} $\cM$ has no direct summands of the form $\mathrm M_n(\cN)$ for $n\geq 1$ and an infinite-dimensional abelian von Neumann algebra $\cN$. 
\item \label{tn2} Every representation of $\cM$ on a separable Hilbert space is automatically normal. 
\item \label{tn3} Every representation of $\cM$ on a finite-dimensional Hilbert space is automatically normal.
\end{enumerate} 
\end{theorem}

\begin{proof}
Assume \eqref{tn1}. As in the proof of Proposition \ref{PropvNaStruc} we may write $\cM$ as a direct sum 
\[
\cM=\cM_0\oplus \cM_1
\]
where $\cM_0$ is a direct product of von Neumann algebras of the form $\mathrm{M}_n(\cN_n)$ for some abelian von Neumann algebra $\cN_n$, and $\cM_1$ contains no summand of finite type I. Using \cite[Theorem~V.5.1]{Tak:TheoryI}, every representation of $\cM_1$ on a separable Hilbert space is automatically normal, so we may assume $\cM=\cM_0$. Moreover, assumption \eqref{tn1} tells us that each $\cN_n$ must be finite-dimensional (possibly zero). Hence either $\cM$ is finite-dimensional (in which case any representation at all is normal), or $\cM$ is infinite-dimensional and of the form $\prod_k \mathrm M_{n_k}(\bbC)$ for a sequence $(n_k)_k$ that converges to infinity. Theorem \ref{takemoto the} then gives us condition \eqref{tn2}.

The implication from condition \eqref{tn2} to \eqref{tn3} is trivial, so it remains to show that \eqref{tn3} implies \eqref{tn1}. Assume that \eqref{tn1} fails. We first consider the case when $\cM$ is abelian; we may assume moreover that $\cM$ is infinite-dimensional, otherwise \eqref{tn1} is trivially true. Let $(p_i)_{i\in I}$ be a maximal collection of mutually orthogonal minimal projections in $\cM$ (possibly empty). Then $\cM\cong \mathcal{D}\oplus \ell_\infty(I)$, where $\mathcal{D}$ is diffuse (possibly zero). If $I$ is infinite, then the $*$-homomorphism $\phi:\ell_\infty(I)\to \C$ defined by evaluation along any non-principal ultrafilter on~$I$ is a non-normal representation on a one-dimensional Hilbert space. If $I$ is finite, $\mathcal{D}$ must be non-zero by infinite-dimensionality of $\cM$. As $\mathcal{D}$ is a non-zero commutative \cstar-algebra, there is a (non-zero) multiplicative linear functional $\phi:\mathcal{D}\to \C$. As $\mathcal{D}$ is diffuse, for any non-zero projection $p\in \mathcal{D}$ there is a non-zero projection $p_0\leq p$ such that $\phi(p)=0$: this follows as we can write $p=q+r$ for two non-zero orthogonal projections $q$ and $r$; as $\phi$ can only take the values $0$ and $1$ on projections, it must send at least one of $q$ and $r$ to zero. Hence if $(p_i)_i$ is a maximal family of orthogonal projections in $\mathcal{D}$ such that $\phi(p_i)=0$, we have that $\sum_i p_1=1_{\mathcal{D}}$ so $\phi(\sum p_i)=1$ even though $\phi(p_i)=0$ for all $i$. Hence $\phi$ is not normal and \eqref{tn3} fails. 

For the general case, suppose that $\cM$ has a direct summand of the form~$\mathrm M_n(\cN)$, where $n\geq 1$ and $\cN$ is an infinite-dimensional abelian von Neumann algebra. The above argument gives a non-normal representation $\phi:\cN\to \C$, and the amplification $\phi\otimes 1_n:\cN\otimes \mathrm{M}_n(\C)\to \mathrm{M}_n(\C)$ is then also not normal.
\end{proof} 

\begin{remark}
We ought to comment on glaring difference between Proposition~\ref{P.SOT} and the stronger statement for $\prod_n \mathrm M_n(\bbC)$ in Theorem \ref{takemoto the} (and other von Neumann algebras that do not have an infinite-dimensional summand of type I) given in the implication from \eqref{tn1} to \eqref{tn2} in Theorem~\ref{T.Normal}. 
In the latter case, every $*$-homomorphism is automatically normal. In the case of $L_\infty(Z,\mu)$, we only claim that if there is an injective $*$-homomorphism, then its restriction to the corner defined by a non-zero projection in the algebra is a nontrivial normal $*$-homomorphism. The implication from \eqref{tn2} to~\eqref{tn1} in Theorem~\ref{T.Normal} shows that this conclusion cannot be improved. 

This is analogous to the situation with embeddings of corona \cstar-algebras. In the abelian case, an abundant supply of nontrivial embeddings (also constructed using ultrafilters, in a manner similar to the proof that \eqref{tn3} implies~\eqref{tn1} in Theorem~\ref{T.Normal}) abound in ZFC while in the case of e.g., the Calkin algebra forcing axioms imply that all endomorphisms are trivial (\cite{Vac:Trivial}). An analogous result conjecturally holds for the coronas of separable \cstar-algebras that are simple \cstar-algebras. 
\end{remark}

\begin{proposition}\label{sep predual}
Let $\cM$ be a von Neumann algebra, and assume there exists a (not necessarily normal and not necessarily unital) faithful representation $\pi:\cM\to \cB(H)$ of $\cM$ on a separable Hilbert space. Then $\cM$ has separable predual.
\end{proposition}

For the proof, recall that the \emph{density character} of a Banach space is the smallest possible cardinality of a dense subset.

\begin{proof}
As in the proof of Proposition \ref{P.SOT}, we may use \cite[Theorem~V.5.1]{Tak:TheoryI} and the structure theory of von Neumann algebras to reduce to the case that $\cM=\prod_{n\in \N}\mathrm{M}_n(\cN_n)$ where each $\cN_n$ is abelian. Hence it suffices to prove the result for all the von Neumann algebras $\cN_n$. Let us assume therefore that $\cM$ is abelian.

For each projection $p\in \cM$, let $\kappa(p)$ be the density character of the predual of $p\cM p$. Let $F$ be a maximal family of orthogonal projections in $\cM$ such that $\kappa(q)=\kappa(p)$ for all non-zero $q\leq p$. Define $p_0:=1_{\cM}-\sum_{p\in F} p$, which we claim is zero. Indeed, if not the definition of $F$ allows us to build a decreasing sequence $p_0\geq p_1\geq p_2\geq \cdots$ of projections in $\cM$ such that $\kappa(p_n)>\kappa(p_{n+1})$ for all $n$. This, however, gives a strictly decreasing infinite sequence of cardinals, which is impossible.

We claim next that for all $p\in F$, $p\cM p$ has separable predual. If this does not hold for some $p$, Proposition \ref{P.SOT} gives non-zero $r\leq p$ such that $\pi$ restricted to $r\cM r$ is normal. However, by definition of $F$, $r\cM r$ is non-separable, and the map $(\pi|_{r\cM r})_*:\cB(H)_*\to (r\cM r)_*$ is onto as $\pi$ is injective, so this is a contradiction. 

Finally, as $\cM$ is abelian, $\cM=\prod_{p\in F} p\cM p$, so we are done.
\end{proof}

\begin{acknowledgments}
This paper was written under the auspices of the American Institute of Mathematics (AIM) SQuaREs program as part of the `Expanders, ghosts, and Roe algebras' SQuaRE project. F.\ B. was partially supported by the US National Science Foundation under the grants DMS-1800322 and DMS-2055604, B.\ M.\ B. was partially supported by the US National Science Foundation under the grant DMS-2054860. I.\ F.\ is partially supported by NSERC. A.\ V.\ is supported by an `Emergence en Recherche' IdeX grant from Universit\'e Paris Cit\'e and an ANR grant (ANR-17-CE40-0026). R.\ W.\ is partially supported by the US National Science Foundation under the grant DMS-1901522.
\end{acknowledgments}

\newcommand{\etalchar}[1]{$^{#1}$}
\providecommand{\bysame}{\leavevmode\hbox to3em{\hrulefill}\thinspace}
\providecommand{\MR}{\relax\ifhmode\unskip\space\fi MR }
\providecommand{\MRhref}[2]{%
  \href{http://www.ams.org/mathscinet-getitem?mr=#1}{#2}
}
\providecommand{\href}[2]{#2}

\end{document}